\documentclass[11pt]{amsart}
\usepackage{amscd,amssymb,euscript,mathrsfs}

\usepackage{color}

\newcounter{noindnum}[subsection]
\renewcommand{\thenoindnum}{\roman{noindnum}}
\newcommand{\noindstep}{\refstepcounter{noindnum}{{\rm(}\thenoindnum}\/{\rm)} }
\newcommand{\stepzero}{\setcounter{noindnum}{0}}

\newcommand{\EuD}{\EuScript D}
\newcommand{\For}{\mathrm{For}}
\newcommand{\Fib}{\mathrm{Fib}}
\newcommand{\Rat}{\mathrm{Rat}}

\renewcommand{\phi}{\varphi}
\renewcommand{\epsilon}{\varepsilon}
\renewcommand{\emptyset}{\varnothing}

\newcommand{\C}{\mathbb C}

\newcommand{\A}{\mathbb A}
\newcommand{\Q}{\mathbb Q}
\newcommand{\Z}{\mathbb{Z}}

\newcommand{\twext}{\mathrel{\tilde\boxtimes}}

\DeclareMathOperator{\Gr}{Gr}
\DeclareMathOperator{\gr}{gr}
\DeclareMathOperator{\supp}{supp}
\DeclareMathOperator{\Id}{Id}
\DeclareMathOperator{\Rep}{Rep}
\DeclareMathOperator{\MHM}{MHM}
\DeclareMathOperator{\Tate}{Tate}
\DeclareMathOperator{\MH}{MH}
\DeclareMathOperator{\FM}{FM}
\DeclareMathOperator{\VHS}{VHS}
\DeclareMathOperator{\DR}{DR}
\DeclareMathOperator{\Coh}{Coh}
\DeclareMathOperator{\Vect}{Vect}
\DeclareMathOperator{\Perv}{Perv}

\DeclareMathOperator{\Spec}{Spec}

\DeclareMathOperator{\Hom}{Hom}

\DeclareMathOperator{\St}{St}
\DeclareMathOperator{\Ext}{Ext}

\newcommand{\IC}{\mathrm{IC}}
\newcommand{\gm}[1]{\mathbf{G_{m,#1}}}

\newcommand{\cC}{\mathcal C}

\newcommand{\cE}{\mathcal E}

\newcommand{\cH}{\mathcal H}
\newcommand{\cM}{\mathcal M}

\newcommand{\cO}{\mathcal O}
\newcommand{\cK}{\mathcal K}

\newcommand{\cS}{\mathcal S}

\newcommand{\cF}{\mathcal F}
\newcommand{\cG}{\mathcal G}
\newcommand{\cV}{\mathcal V}

\theoremstyle{plain}
\newtheorem{theorem}{Theorem}
\newtheorem{proposition}{Proposition}[section]
\newtheorem{lemma}[proposition]{Lemma}

\theoremstyle{definition}
\newtheorem{definition}{Definition}

\theoremstyle{remark}
\newtheorem*{remark}{Remark}

\title[Satake equivalence for Hodge modules]{Satake equivalence for Hodge modules on affine Grassmannians}

\author{Roman~Fedorov}
\address{University of Pittsburgh, USA}
\email{fedorov@pitt.edu}

\begin{document}

\begin{abstract}
For a reductive group $G$ we equip the category of $G_\cO$-equivariant polarizable pure Hodge modules on the affine Grassmannian $\Gr_G$ with a structure of neutral Tannakian category. We show that it is equivalent to a twisted tensor product of the category of representations of the Langlands dual group and the category of pure polarizable Hodge structures.
\end{abstract}

\maketitle

\section{Introduction and Main Results}
We fix a complex reductive group $G$. Consider the affine Grassmannian $\Gr_G$ defined as the \'etale sheafification of the functor sending a $\C$-algebra $R$ to the quotient $G\bigl(R((t))\bigr)/G(R[[t]])$, where $R((t)):=R[[t]](t^{-1})$. Let $G_\cO$ be the algebraic group of infinite type representing the functor $R\mapsto G(R[[t]])$, that is, the group of jets. Clearly, $G_\cO$ acts on~$\Gr_G$.

It is well-known that the affine Grassmannian is a strict ind-projective ind-scheme, that is, we can write
\begin{equation}\label{eq:GrIndSch}
    \Gr_G=\lim_{l\to\infty}\Gr_G^{(l)},
\end{equation}
where $(\Gr_G^{(l)})$ is an inductive systems of projective schemes with morphisms being closed embeddings. We may assume that the schemes $\Gr_G^{(l)}$ are $G_\cO$-equivariant.

\subsection{Recollection on the Geometric Satake Equivalence} The first work in the direction of geometrizing the Satake isomorphism was Lusztig's paper~\cite{Lusztig1983singularities}. Independently, Drinfeld understood that geometrizing the Satake isomorphism is crucial for the geometric Langlands correspondence. Following Drinfeld's suggestion, Ginzburg in~\cite{Ginzburg95Satake} proved the characteristic zero case of the Geometric Satake equivalence. We follow~\cite{MirkovicVilonen}, where Mirkovic and Vilonen generalized the above results.

In more details, Mirkovic and Vilonen equip the category $\Perv_\Q^{G_\cO}(\Gr_G)$ with a convolution product and a commutativity constraint for this product, making it into a symmetric monoidal category. They also construct a fiber functor making $\Perv_\Q^{G_\cO}(\Gr_G)$ into a neutral Tannakian category. Thus, according to Tannakian formalism (see for example~\cite[Thm.~2.11]{DeligneMilneTannakian}), there is an affine group $\check G_\Q$ over $\Q$ such that we have an equivalence of categories
\[
    \Perv_\Q^{G_\cO}(\Gr_G)=\Rep_\Q(\check G_\Q).
\]
Here $\Rep_\Q$ stands for the category of finite dimensional representations on $\Q$-vector spaces. The main result of~\cite{MirkovicVilonen} shows that $\check G_\Q$ is split reductive and is Langlands dual to $G$. The main goal of our paper is to ``upgrade'' this theorem from perverse sheaves to pure polarizable Hodge modules.

Recently the geometric Satake equivalence was upgraded to the motivic setup in~\cite{RicharzScholbach2021motivicSatake}. It seems plausible that some of our results can be obtained from Hodge realization of the results in loc.~cit. It is, however, important to have a direct construction. Our result is also, in a sense, simpler, as it does not use the theory of motives.

\subsection{Geometric Satake Equivalence for pure polarizable Hodge modules}
For a complex variety~$Y$ (=reduced separated scheme of finite type over $\C$) one has the category $\MH(Y,n)=\MH(Y,\Q,n)$ of pure polarizable rational Hodge modules of weight $n$ (cf.~\cite{SaitoPolarizable,SaitoMixed}). In~Section~\ref{sect:BlBox} we extend the notion of Hodge modules to ind-schemes and to the equivariant setting so that we have the category $\MH^{G_\cO}(\Gr_G,n)$ of $G_\cO$-equivariant pure polarizable Hodge modules of weight $n$ on $\Gr_G$. This is a semisimple abelian $\Q$-linear category. Consider the graded category
\[
    \MH^{G_\cO}(\Gr_G):=\bigoplus_{n\in\Z}\MH^{G_\cO}(\Gr_G,n).
\]
We equip this category with a structure of neutral Tannakian category in Section~\ref{sect:tannakian}.

Let $\MH(pt,n)$ stand for the category of polarizable Hodge structures of weight~$n$ (which is the same as the polarizable Hodge modules of weight $n$ on a point, thus the notation; see Section~\ref{sect:HodgeStrict}) and set
\[
    \MH(pt):=\bigoplus_{n\in\Z}\MH(pt,n).
\]
Recall that we have the Tate twist, which is a degree $-2$ autoequivalence of $\MH(pt)$.

Our main goal in this paper is to identify the category $\MH^{G_\cO}(\Gr_G)$ as a ``twisted'' tensor product. We will need two constructions.
\subsubsection{Adding the square root of the Tate twist}\label{sect:UntwistTate} Let $\cC$ be a category with an auto-equivalence $T$. We define the new category $\cC[\sqrt T]:=\cC^+\times\cC^-$ to be the product of two copies of $\cC$. This decomposition is viewed as a $\Z/2\Z$ grading. If $\cF$ is an object of $\cC$, we write $\cF^\pm$ for $\cF$ viewed as an object of $\cC^\pm\subset\cC[\sqrt T]$.

We define an autoequivalence $\sqrt T$ of $\cC[\sqrt T]$ by
\[
    \sqrt T(\cF^+)=\cF^-,\quad \sqrt T(\cF^-)=T(\cF)^+.
\]
Note that $T=(\sqrt T)^2$. We will also write $\cF(n/2)$ for $\sqrt T^n(\cF)$.

Assume that $\cC=\bigoplus_{i\in\Z}\cC^i$ is a graded additive category and that $T$ has degree $-2$. Then the category $\cC[\sqrt T]$ also has a $\Z$-grading defined by $\cC[\sqrt T]^i:=(\cC^+)^i\times(\cC^-)^{i+1}$.
Note that $\sqrt T$ has degree $-1$.

Assume now that $\cC$ is a monoidal category and $T$ is compatible with the monoidal structure in the sense that we have natural isomorphisms
\[
    T(\cF\otimes\cG)=T(\cF)\otimes\cG=\cF\otimes T(\cG)
\]
for any objects $\cF$ and $\cG$. Then we define a monoidal structure on $\cC[\sqrt T]$ by
\[
    \cF^+\otimes\cG^+:=(\cF\otimes\cG)^+,\quad
    \cF^+\otimes\cG^-:=(\cF\otimes\cG)^-,\quad
    \cF^-\otimes\cG^+:=(\cF\otimes\cG)^-,\quad
    \cF^-\otimes\cG^-:=(T(\cF\otimes\cG))^+
\]
The product is compatible with $\sqrt T$. If $\cC$ is $\Z$-graded, then this product is compatible with gradings.

\subsubsection{Tensor product of semisimple categories}
Recall that for abelian semisimple $\Q$-linear categories $\cC_1$ and $\cC_2$ one can define their tensor product $\cC_1\boxtimes\cC_2$ (see, e.g.~\cite[Def.~1.1.15]{KirillovBakalov}). The objects of this category are formal expressions $\oplus_{i=1}^n V_{1,i}\boxtimes V_{2,i}$, where $V_{j,i}$ are objects of $\cC_j$; the morphisms are given by
\[
    \Hom(\oplus_{i=1}^n V_{1,i}\boxtimes V_{2,i},\oplus_{j=1}^m V'_{1,j}\boxtimes V'_{2,j}):=
    \oplus_{i,j}\Hom(V_{1,i},V'_{1,j})\otimes_\Q\Hom(V_{2,i},V'_{2,j}).
\]
The tensor product of neutral Tannakian categories is again a neutral Tannakian category.

\subsubsection{} Recall that $\check G_\Q$ is the split reductive group over $\Q$ Langlands dual to $G$. Fix a Borel subgroup $\check B_\Q\subset\check G_\Q$ and a split maximal torus $\check T_\Q\subset\check B_\Q$. The sum of all positive coroots of $\check G_\Q$ with respect to $\check T_\Q\subset\check B_\Q$ is a cocharacter $2\rho\colon\gm\Q\hookrightarrow\check G_\Q$.

The simple objects of $\Rep_\Q\check G_\Q$ correspond to the dominant weights of $\check G_\Q$. If $\lambda$ is such a weight, we call the corresponding representation even or odd depending on the parity of $2\rho(\lambda)$. Thus, we equip the category $\Rep_\Q\check G_\Q$ with a $\Z/2\Z$ grading. Now we can formulate our first main result.
\begin{theorem}\label{th:HodgeSatake}
The category $\MH^{G_\cO}(\Gr_G)$ has a natural structure of a neutral Tannakian category. There is a canonical equivalence of Tannakian categories compatible with the gradings and the fiber functors
  \[
    \MH^{G_\cO}(\Gr_G)=\left(\Rep_\Q\check G_\Q\boxtimes\MH(pt)[\sqrt T]\right)^+,
  \]
  where $\MH(pt)[\sqrt T]$ is an instance of the construction defined in Section~\ref{sect:UntwistTate}, $T$ is the Tate twist; the superscript $+$ stands for the even component of the $\Z/2\Z$-graded category.
\end{theorem}

We note that if $-1$ is in the kernel of $2\rho\colon\gm\Q\hookrightarrow\check G_\Q$, then the $\Z/2\Z$-grading on $\Rep_\Q\check G_\Q$ is trivial, and the previous theorem becomes
\[
    \MH^{G_\cO}(\Gr_G)=\Rep_\Q\check G_\Q\boxtimes\MH(pt).
\]

\subsection{Geometric Satake Equivalence for Tate's Hodge modules}
Let $\Tate^{G_\cO}(\Gr_G)$ be the strictly full subcategory of $\MH^{G_\cO}(\Gr_G)$ whose objects are isomorphic to direct sums of Tate twists of IC Hodge modules on the $G_\cO$-orbits in $\Gr_G$. This is a semisimple monoidal subcategory of $\MH^{G_\cO}(\Gr_G)$, see Proposition~\ref{pr:ClosedUnderStar} below. We can give a more direct Tannakian description of this category.

We can form a twisted product
\[
   \check G_\Q\times^{\mu_2}\gm\Q:=(\check G_\Q\times\gm\Q)/((2\rho\times\iota)(\mu_2)),
\]
where $\iota\colon\mu_2\to\gm\Q$ is the standard embedding. This is Deligne's modified Langlands dual group.

\begin{theorem}\label{th:TateSatake}
There is a canonical equivalence of Tannakian categories compatible with the gradings and the fiber functors
    \[
        \Tate^{G_\cO}(\Gr_G)\simeq\Rep_\Q(\check G_\Q\times^{\mu_2}\gm\Q).
    \]
\end{theorem}

This should be compared with~\cite[Thm.~C]{RicharzScholbach2021motivicSatake}.

\subsection{Geometric Satake equivalence for mixed Hodge modules} Similarly to the above, one can define the category $\MHM^{G_\cO}(\Gr_G)$ of $G_\cO$-equivariant mixed Hodge modules on $\Gr_G$. In Section~\ref{sect:Mixed} we sketch the proof of the following theorem.

\begin{theorem}\label{th:MixedSatake}
The category $\MHM^{G_\cO}(\Gr_G)$ has a natural structure of a neutral Tannakian category. There is a canonical equivalence of Tannakian categories compatible with the fiber functors
    \[
        \MHM^{G_\cO}(\Gr_G)\simeq\left(\Rep_\Q\check G_\Q\boxtimes\MHM(pt)[\sqrt T]\right)^+,
    \]
where, as before, $T$ is the Tate twist and the superscript $+$ stands for the even component of the $\Z/2\Z$-graded category.
\end{theorem}

We note that the category $\MHM(pt)$ is not semisimple. However, the tensor product of two $\Q$-linear categories also makes sense if only one of the categories is semisimple. Thus, the right hand side of the above equivalence of categories makes sense.

\subsection{Future directions} Consider the cotangent stack
\[
    \St_G:=T^*(G_\cO\backslash\Gr_G),
\]
which we call the \emph{Steinberg Stack}. Here the stack $G_\cO\backslash\Gr_G$ is the quotient of $\Gr_G$ by the action of $G_\cO$. Note that $\St_G$  is \emph{not} an ind-algebraic stack, since some points have infinite-dimensional automorphism groups. However, one can show that it is a quotient of an ind-scheme by $G_\cO$. In a subsequent work we will construct
a functor
\[
    \Rep_\Q(\check G_\Q\times^{\mu_2}\gm\Q)\to\Coh(\St_G),
\]
where $\Coh(\bullet)$ is the abelian category of coherent sheaves. We will equip $\Coh(\St_G)$ with a monoidal structure and show that the functor is uniquely defined by certain axioms, provided that the quotient of $G$ by its radical is of adjoint type. We call this functor the \emph{semi-classical Satake functor}. We expect that this functor will play a crucial role in the local Langlands duality for Hitchin systems. Let us sketch the construction of this functor.

Assume first that $Y$ is a separated scheme smooth over $\C$. Then $\MH(Y)$ is a strictly full subcategory of the category of $\EuD$-modules on $Y$ with good filtrations (see Section~\ref{sect:HodgeOnSch} below). Since the associated graded of a filtered $\EuD$-module is a coherent sheaf on the cotangent bundle, we get a functor $\gr\colon\MH(Y)\to\Coh(T^*Y)$. This can be extended to ind-schemes and to the equivariant setting so that we get a functor
\[
    \gr\colon\MH^{G_\cO}(\Gr_G)\to\Coh(\St_G).
\]
The semi-classical Satake functor is defined as the composition
\[
    \Rep_\Q(\check G_\Q\times^{\mu_2}\gm\Q)\xrightarrow{\simeq}\Tate^{G_\cO}(\Gr_G)\hookrightarrow\MH^{G_\cO}(\Gr_G)
    \xrightarrow{\gr}\Coh(\St_G).
\]
The biggest difficulty in defining this functor is that $\Gr_G$ is not a limit of smooth projective schemes.

\subsection{Notation and conventions} We will be working with left $\EuD$-modules. We will freely work with perverse sheaves on ind-schemes. The category of perverse sheaves of vector spaces over $\Q$ on $Y$, where $\Q$ is a field, is denoted by $\Perv_\Q(Y)$.

We call an affine group scheme over $\Q$ a group for brevity.

We write $\simeq$ for an isomorphism and simply $=$ for a canonical isomorphism.

\subsection{Organization of the paper} In Section~\ref{sect:BlBox} we recall all the statements we need about Hodge modules. We extend the definition of Hodge modules to the case of ind-schemes and to the equivariant setup.

In Section~\ref{sect:tannakian} we equip the category $\MH^{G_\cO}(\Gr_G)$ with a structure of neutral Tannakian category. In Section~\ref{Sect:proof} we prove Theorems~\ref{th:HodgeSatake} and~\ref{th:TateSatake}. Finally, in Section~\ref{sect:Mixed} we sketch a proof of Theorem~\ref{th:MixedSatake}.

\subsection{Acknowledgements} A large part of this paper was written while the author visited Max Planck Institute for Mathematics in Bonn (July 2018) and also Research Institute for Mathematical Sciences in Kyoto (December 2018). The author would like to thank Claude Sabbah for answering numerous questions about Hodge modules. The author also benefited from talks with Roman Bezrukavnikov. The author is thankful to Dennis Gaitsgory for the idea of the proof of Proposition~\ref{pr:ClosedUnderStar} and to George Lusztig for correcting a historical mistake. Of great importance was advice from Dima Arinkin. The paper would not have been written without his help.

The author is partially supported by NSF grants DMS--1406532, DMS--2001516.

\section{Hodge modules on ind-schemes}\label{sect:BlBox} The categories of Hodge modules on varieties have been defined and studied by M.~Saito in~\cite{SaitoPolarizable,SaitoMixed}. In this section we extend this to the case of ind-separated ind-schemes with `nice' actions of affine groups. We formulate all the statements we will need later and prove most of them (leaving a few simpler statements to the reader). It seems that some of the results about pure polarizable Hodge modules (especially the results related to the inverse images) can only be proved with the help of mixed Hodge modules, so we will use mixed Hodge modules in the proofs as well.

\subsection{Hodge modules on schemes} We briefly recall M.~Saito's theory of Hodge modules.
\subsubsection{Filtered $\EuD$-modules with rational structures on smooth schemes} Let $Y$ be a separated scheme smooth over $\C$. Let $\EuD_Y$ denote the sheaf of differential operators on $Y$ with the usual filtration $\EuD^{\le p}_Y$ by the degree. Let $M$ be a $\EuD_Y$-module. Recall that an increasing filtration $F_\bullet$ of $M$ by $\cO_Y$-coherent submodules is \emph{good\/} if it is separated and exhaustive, it satisfies $\EuD^{\le p}_Y\cdot F_qM\subset F_{p+q}M$, and there is $q_0\in\Z$ such that for all $p\ge0$ and $q\ge q_0$ we have $\EuD^{\le p}_Y\cdot F_qM=F_{p+q}M$.

Let $\FM_{rs}(\EuD_Y)$ be the category of quadruples $(M,F_\bullet,\cM,\alpha)$, where $M$ is a holonomic $\EuD_Y$-module with regular singularities, $F_\bullet$ is a good filtration on $M$, $\cM\in\Perv_\Q(Y)$, and $\alpha$ is an isomorphism $\alpha\colon\DR(M)\xrightarrow{\simeq}\cM\otimes_\Q\C$ ($\DR$ is the de Rham functor).

Since $\DR$ is an equivalence of categories, we see that $\FM_{rs}(\EuD_Y)$ can be defined equivalently as the category of pairs $(\cM,F_\bullet)$, where $\cM\in\Perv_\Q(Y)$, $F_\bullet$ is good filtration of $\DR^{-1}(\cM\otimes_\Q\C$). It is clear that $\FM_{rs}(\EuD_Y)$ is an additive $\Q$-linear category. We define the functor $\Rat\colon\FM_{rs}(\EuD_Y)\to\Perv_\Q(Y)$ by $\Rat(M,F_\bullet,\cM,\alpha)=\cM$. We see that $\Rat$ is faithful.


If $j\colon U\to Y$ is an open embedding, we have an obvious restriction functor
\begin{equation}\label{eq:resopen}
    j^*\colon\FM_{rs}(\EuD_Y)\to\FM_{rs}(\EuD_U).
\end{equation}
On the other hand, let $i\colon Y'\to Y$ be a closed embedding of smooth varieties with connected $Y'$. Following~\cite[Intro.]{SaitoPolarizable}, we define the direct image functor
\begin{equation}\label{eq:closedembedding}
    i_*\colon\FM_{rs}(\EuD_{Y'})\to\FM_{rs}(\EuD_Y)
\end{equation}
by
\[
    i_*(M,F_\bullet,\cM,\alpha):=(i_*M,{'\!F}^\bullet,i_*\cM,i_*(\alpha)).
\]
We are using the fact that $\DR$ commutes with $i_*$. The filtration $'\!F_\bullet$ is defined as follows. Recall that $i_*M=i_{ab,*}(\EuD_{Y\leftarrow Y'}\otimes_{\EuD_{Y'}}M)$, where $i_{ab,*}$ is the direct image for sheaves of abelian groups.  Then
\[
    {'\!F}^pi_*M:=i_{ab,*}\left(\sum_{q+\nu\le p-n}F_\nu\EuD_{Y\leftarrow Y'}\otimes F_qM\right),
\]
where $n$ is the codimension of $Y'$ in $Y$. It is clear that the restriction functor for the composition of open embeddings is the composition of restriction functors, and the direct image functor for the composition of closed embeddings is the composition of direct image functors. It is also clear that $i_*$ commutes with restriction to open subsets (this is a particular case of the base change). We define the support of a filtered $\EuD_Y$-module with a rational structure by
\[
    \supp(M,F_\bullet,\cM,\alpha):=\supp(M)=\supp(\cM).
\]
Thus $\supp(M,F_\bullet,\cM,\alpha)$ is a reduced closed subscheme of $Y$.

\subsubsection{Polarizable Hodge modules on separated schemes of finite type}\label{sect:HodgeOnSch}
In~\cite{SaitoPolarizable} for every smooth over $\C$ separated scheme $Y$ Saito defines the category $\MH(Y,\Q,n)^p$ of polarizable Hodge module of weight $n$. This is a strictly full subcategory of $\FM_{rs}(\EuD_Y)$. More precisely, Saito works with an analogue of $\FM_{rs}(\EuD_Y)$, where the $\EuD_Y$-modules are only assumed to be holonomic. However, in retrospect all $\EuD_Y$-modules underlying polarizable Hodge modules have regular singularities~\cite[p.~858 and Rem.~5.1.18]{SaitoPolarizable}.

\emph{Convention}. We will only work with rational polarizable Hodge modules, so we denote $\MH(Y,\Q,n)^p$ by $\MH(Y,n)$ and sometimes skip the adjective `polarizable'.

Let now $Y$ be a not necessarily smooth reduced separated scheme of finite type over $\C$. In~\cite[Sect.~5.3.12]{SaitoPolarizable} (see also~\cite[Sect.~2.1]{SaitoMixed}) Saito extends the definition of $\MH(Y,n)$ to this case as follows. Let $Y=\bigcup_i U_i$ be a Zariski cover and $U_i\to V_i$ be closed embeddings into smooth schemes. Then a Hodge module on $Y$ is a collection $\cF_i$ of Hodge modules on $V_i$ such that $\supp(\cF_i)\subset U_i$ together with isomorphisms on intersections satisfying the cocycle condition. The category $\MH(Y,n)$ does not depend on the cover and the embeddings because Hodge modules on smooth separated schemes satisfy the Zariski descent and Kashiwara's equivalence.

\begin{remark}
Polarizable Hodge modules satisfy Kashiwara's equivalence according to~\cite[Lemme~5.1.9, Lemme~5.2.11]{SaitoPolarizable}. It follows implicitly from~\cite[Sect.~5.3.12]{SaitoPolarizable} that they satisfy Zariski descent. The reason is, however, not quite clear. In fact, Zariski descent follows from definitions for Hodge modules, but it is not obvious that a locally polarizable module is polarizable. This, however, follows from~\cite[Thm.~3.21]{SaitoMixed} after reducing to the case of a module with strict support (cf.~also Proposition~\ref{pr:Saito3.21} below).
\end{remark}

To summarize, for a reduced separated $\C$-scheme of finite type, we have the category of  Hodge modules; the Hodge modules satisfy the Zariski descent and Kashiwara's equivalence. In particular, the restriction functors~\eqref{eq:resopen} to open subsets of smooth schemes preserve the subcategories of Hodge modules and the definition of these functors easily extends to the case of singular schemes. The same applies to the direct image functors~\eqref{eq:closedembedding}.

We can easily extend the definition of Hodge modules to non-reduced schemes by setting $\MH(Y,n):=\MH(Y_{red},n)$. Kashiwara's equivalence and the Zariski locality extend to this case. One also extends the definition of the functor $\Rat\colon\MH(Y,n)\to\Perv_\Q(Y)$ to not necessarily reduced schemes. Note that the support of a Hodge module is always reduced.

\begin{remark}
We can also extend the definition of Hodge modules to not necessarily separated schemes locally of finite type over $\C$ using Zariski locality. However, we will not need it. In fact, all our schemes will be quasi-projective. Thus, we leave the generalization to the reader.
\end{remark}

\subsubsection{Hodge modules with strict support and polarizable variations of Hodge structures}\label{sect:HodgeStrict}
Let $Y$ be any separated scheme of finite type over $\C$. Let $Z\subset Y$ be an integral closed subscheme. Recall from~\cite[Sect.~2.1]{SaitoMixed} that $\cF\in\MH(Y,n)$ \emph{has strict support\/} $Z$ if $\supp\cF=Z$ and $\Rat(\cF)$ is an IC-sheaf (that is, $\Rat(\cF)$ has neither quotient objects nor subobjects supported on a proper subscheme of $Z$, cf.~also~\cite[p.~3]{SaitoPolarizable}). Denote the full subcategory of Hodge modules of weight $n$ with strict support $Z$ by $\MH_Z(Y,n)$. If $Z$ is an integral separated scheme of finite type over $\C$ we set $\MH_Z(n):=\MH_Z(Z,n)$.

Let $U$ be an integral smooth scheme. Recall that a rational variation of Hodge structures (VHS for short) of weight $n$ on $U$ is a pair $(\cS,F_\bullet)$, where $\cS$ is a locally constant sheaf of $\Q$-vector spaces, $F_\bullet$ is a filtration of $\cS\otimes_\Q\cO_U$ satisfying Griffiths transversality condition and such that for all $u\in U$ the pair $(\cS_u,(F_\bullet)_u)$ is a Hodge structure of weight $n$. We note that in this case $\cS[\dim U]$ is a perverse sheaf and $\DR^{-1}(\cS[\dim U]\otimes_\Q\C)=(\cS\otimes_\Q\cO_U,\nabla)$, where $\nabla$ is the connection whose sheaf of flat sections is $\cS\otimes_\Q\C$. Then $F_\bullet$ is a good filtration of $(\cS\otimes_\Q\cO_U,\nabla)$ and we see that a VHS is a particular case of a filtered $\EuD$-module with rational structure.

Denote by $\VHS(U,n)$ the category of polarizable variations of Hodge structure of weight $n$ with coefficients in $\Q$ with quasi-unipotent local monodromies. If $W\subset U$ is an open subscheme, then we have a restriction functor
\[
    \VHS(U,n)\to\VHS(W,n).
\]
For an integral separated scheme $Z$ of finite type over $\C$, let $\VHS_{gen}(Z,n)$ denote the direct limit of categories $\VHS(U,n)$, where $U$ ranges over smooth open non-empty subschemes of $Z$.

\begin{proposition}\label{pr:Saito3.21} Let $Z$ be an integral separated scheme of finite type over $\C$. Then there is a canonical equivalence of categories
\[
    \VHS_{gen}(Z,n-\dim Z)\xrightarrow{\simeq}\MH_Z(n)
\]
satisfying the following compatibilities\\
(i) If $Z$ is smooth, then this equivalence associates to a VHS on $Z$ the corresponding filtered $\EuD$-module with rational structure.\\
(ii) If $j\colon Z'\subset Z$ is an open embedding of a non-empty subscheme, then we have a canonical commutative diagram  of functors
\[
\begin{CD}
    \VHS_{gen}(Z,n-\dim Z)@>\simeq>>\MH_Z(n)\\
    @| @VV j^* V\\
    \VHS_{gen}(Z',n-\dim Z')@>\simeq>>\MH_{Z'}(n)
\end{CD}
\]
In particular, $j^*$ is an equivalence of categories.
\end{proposition}
\begin{proof}
  The construction of this equivalence of categories is~\cite[Thm.~3.21]{SaitoMixed}. The compatibility with restriction to open subsets is clear from the proof.
\end{proof}

In particular, $\MH(pt,n)$ is the category of rational polarizable Hodge structures of weight~$n$. Note that the functor $\Rat$ associates to a Hodge structure, viewed as an object of $\MH(pt,n)$, the underlying vector space.


\subsection{Hodge modules on ind-schemes} All the ind-schemes we consider are strict ind-separated ind-schemes of ind-finite type over $\C$. That is, we work with \'etale sheaves $Y\colon\mathrm{Aff}/\C\to\mathrm{Sets}^{op}$ representable as $Y=\lim_{l\to\infty} Y^{(l)}$, where $Y^{(l)}$ are separated schemes of finite type over $\C$, and for $l<l'$ the morphism $Y^{(l)}\to Y^{(l')}$ is a closed embedding. Recall that any morphism $Z\to Y$, where $Z$ is a scheme of finite type over $\C$, factors through some $Y^{(l)}$.

We define
\[
    \MH(Y,n):=\lim_{l\to\infty}\MH(Y^{(l)},n),
\]
where the direct limit is with respect to the system of direct image functors corresponding to closed embeddings $Y^{(l)}\to Y^{(l')}$, $l<l'$. By Kashiwara's equivalence, this category does not depend on the presentation of $Y$ as a limit of schemes. Explicitly, a Hodge module on $Y$ is a pair $(l,\cF)$, where $\cF$ is a Hodge module on $Y^{(l)}$.


We note that the definitions of the functor of restriction to an open subscheme and of the direct image under a closed embedding extend easily to the case of ind-schemes. One also extends the definition of the functor $\Rat$ to the case of ind-schemes.

If $Y$ is a scheme viewed as an ind-scheme, then $\MH(Y,n)$ defined above is the same as $\MH(Y,n)$ discussed in Section~\ref{sect:HodgeOnSch}.

We easily extend the notion of support to Hodge modules on ind-schemes. We emphasize that the support of a Hodge module is a reduced subscheme of $Y$ (not an ind-subscheme). The notion of a Hodge module with strict support also extends to the case of ind-schemes. Kashiwara's equivalence gives a canonical identification $\MH_Z(Y,n)=\MH_Z(n)$, whenever $Z$ is an integral subscheme of $Y$. We have the following decomposition
\begin{equation}\label{eq:support_decomposition}
  \MH(Y,n)=\bigoplus_{Z\subset Y}\MH_Z(Y,n)\simeq\bigoplus_{Z\subset Y}\MH_Z(n),
\end{equation}
where the direct sum is over integral closed subschemes of $Y$. Indeed, if $Y$ is a smooth scheme, then this follows from~\cite[p.~852--853]{SaitoPolarizable}. If $Y$ can be embedded into a smooth scheme, this follows from Kashiwara's equivalence. Thus, if $Y$ is any reduced scheme, then the decomposition exists locally. But the decomposition of a Hodge module into the direct sum of modules with strict support is unique because there are no morphisms between Hodge modules with different strict supports. Thus the global decomposition follows from the local one. Finally, the case of ind-schemes is reduced to the case of reduced schemes using again Kashiwara's equivalence.

\begin{proposition}\label{pr:HM_basic_properties}
Let $Y$ be a strict ind-scheme ind-separated and of ind-finite type over $\C$. Then

\stepzero\noindstep\label{basic:i} $\MH(Y,n)$ is a $\Q$-linear abelian semisimple category.

\noindstep\label{basic:ii} The functor $\Rat\colon\MH(Y,n)\to\Perv_\Q(Y)$ is exact, faithful, and conservative. It is compatible with the direct image functor for closed embeddings and with restrictions to open sub-ind-schemes.
\end{proposition}
\begin{proof}
\eqref{basic:i} $\Q$-linearity is obvious from the definition. Let $Z$ be a closed integral subscheme of $Y$. Then according to~\cite[Lemme~5]{SaitoPolarizable} $\MH_Z(n)$ is abelian and semisimple. It remains to use~\eqref{eq:support_decomposition}.

\eqref{basic:ii} The compatibility with closed and open embeddings is clear. It is exact because every additive functor from a semisimple category is exact. To show that the functor is faithful, we first reduce the statement to the case when $Y$ is a smooth scheme and then use the fact that $\MH(Y,n)$ is a full subcategory of $\FM_{rs}(\EuD_Y)$. Finally, any faithful functor between abelian categories is conservative, since abelian categories are balanced.
\end{proof}

\subsection{Derived categories of polarizable Hodge modules}\label{sect:DerHodge} For a reduced separated scheme $Y$ of finite type over $\C$ Saito defines the abelian category of mixed polarizable Hodge modules $\MHM(Y)$ (\cite[(4.2.3)]{SaitoMixed}) and its bounded derived category $D^b\MHM(Y)$. By~\cite[2.17.5]{SaitoMixed} mixed Hodge modules satisfy Kashiwara's equivalence. By~\cite[4.2.10]{SaitoMixed} the derived category $D^b\MHM(Y)$ satisfies Kashiwara's equivalence as well. Thus we can extend the definitions of the categories to the case, when $Y$ is an ind-scheme. We still have a faithful exact functor $\Rat\colon\MHM(Y)\to\Perv_\Q(Y)$.

Every mixed Hodge module has a weight filtration whose associated graded is a pure Hodge module. The category $\MH(Y,n)$ is a full subcategory of $\MHM(Y)$.


We define the derived category $D^b\MH(Y,n)$ of pure Hodge modules of weight $n$ as the full subcategory of $D^b\MHM(Y)$ consisting of complexes $\cF^\bullet$ such that for all $i\in\Z$ we have $H^i(\cF^\bullet)\in\MH(Y,n+i)$. We identify $\MH(Y,n)$ with the full subcategory of $D^b\MH(Y,n)$ consisting of complexes whose cohomology is concentrated in degree~0.

\begin{proposition}\label{pr:KashDb}
(i) The functor
\begin{equation}\label{eq:KashDb}
    \bigoplus_{i\in\Z}\MH(Y,n+i)\to D^b\MH(Y,n)\colon\bigoplus_{i\in\Z}\cF_i\mapsto\bigoplus_{i\in\Z}\cF_i[-i]
\end{equation}
is faithful and essentially surjective. A right inverse functor is given by
\[
    D^b\MH(Y,n)\to\bigoplus_{i\in\Z}\MH(Y,n+i)\colon\cF^\bullet\mapsto\bigoplus_{i\in\Z} H^i(\cF^\bullet).
\]
(ii) Let $i\colon Y\to Y'$ be a closed embedding. Then there is a fully faithful canonical functor $i_*\colon D^b\MH(Y,n)\to D^b\MH(Y',n)$ compatible with functors from part~(i) whose essential image consists of complexes whose cohomology are supported on $Y$.
\end{proposition}

\begin{proof}
  The essential surjectivity in part~(i) follows easily from~\cite[(4.5.4)]{SaitoMixed} 
  in the case when $Y$ is a reduced scheme. It easily extends to the case of ind-schemes. The rest of the proposition follows from the definitions.
\end{proof}

\begin{remark}
Note that functor~\eqref{eq:KashDb} is not full. Indeed, let a mixed Hodge module $\cF$ on $Y$ be an extension of a pure Hodge module $\cF^1$ of weight 1 by a pure Hodge module $\cF^0$ of weight $0$. Then we have in $D^b\MH(Y,1)$: $\Hom(\cF^1,\cF^0[1])=\Ext^1(\cF^1,\cF^0)\ne0$.
\end{remark}

\begin{proposition}\label{pr:DerRat}
  There is a canonical functor $\Rat\colon D^b\MH(Y,n)\to D^b\Perv_\Q(Y)$. This functor is exact and conservative. It is also compatible with the cohomology functor and $\Rat\colon\MH(Y,n)\to\Perv_\Q(Y)$.
\end{proposition}
\begin{proof}
  The proposition is easily reduced to the case, when $Y$ is a reduced scheme. We define the functor by restricting the exact functor $\Rat\colon D^b\MHM(Y)\to D^b\Perv_\Q(Y)$ from~\cite[Thm.~0.1]{SaitoMixed}.

  Let $f$ be a morphism in $D^b\MH(Y)$ such that $\Rat(f)$ is an isomorphism. Then $\Rat(H^i(f))=H^i(\Rat(f))$ is an isomorphism for all $i$. By Proposition~\ref{pr:HM_basic_properties}\eqref{basic:ii} $H^i(f)$ is an isomorphism. Thus $f$ is an isomorphism. We see that $\Rat$ is conservative. The compatibility with cohomology follows from definitions.
\end{proof}

\subsection{The intermediate extension functor}
Let $j\colon Y'\to Y$ be a locally closed embedding and let $Z\subset Y'$ be an integral closed subscheme. Then we have by Kashiwara's equivalence and Proposition~\ref{pr:Saito3.21}(ii) an equivalence
\begin{equation*}
    j_{!*}\colon\MH_{Z}(Y',n)\xrightarrow{\simeq}\MH_Z(n)
    \xrightarrow{\simeq}\MH_{\overline Z}(n)\xrightarrow{\simeq}\MH_{\overline Z}(Y,n),
\end{equation*}
where $\overline Z$ is the Zariski closure of $Z$ in $Y$. We extend $j_{!*}$ to a fully faithful functor called \emph{intermediate extension}
\[
    j_{!*}\colon\MH(Y',n)\to\MH(Y,n)
\]
using~\eqref{eq:support_decomposition}.

The next lemma follows from definitions and Proposition~\ref{pr:Saito3.21}.
\begin{lemma}\label{lm:j*composition}
(i) If $j\colon Y_1\to Y_2$ is an open embedding, then $j^*$ is left inverse to $j_{!*}$.

(ii) If $j\colon Y_1\to Y_2$ is a closed embedding, then $j_{!*}=j_*$.

(iii) Let $j\colon Y_1\to Y_2$ and $j'\colon Y_2\to Y_3$ be locally closed embeddings. Then $(j'\circ j)_{!*}=j'_{!*}\circ j_{!*}$.
\end{lemma}

We have the standard characterization of intermediate extensions.

\begin{lemma}\label{lm:EqualTo!*} Let $j\colon Y'\to Y$ be an open embedding and let $\cF\in\MH(Y,n)$, then the following are equivalent\\
(i) $\Rat(\cF)$ has no subobjects supported on $Y-Y'$.\\
(ii) $\Rat(\cF)$ has no quotient objects supported on $Y-Y'$.\\
(iii) The Hodge modules $j_{!*}j^*\cF$ and $\cF$ are isomorphic.\\
In this case there is a unique isomorphism $j_{!*}j^*\cF\to\cF$ whose restriction to $Y'$ is the identity.
\end{lemma}
\begin{proof}
It is clear from the definition of $j_{!*}$ that for any $\cF'\in\MH(Y',n)$ the Hodge module $j_{!*}\cF'$ has no sub- or quotient-objects supported on $Y-Y'$. Thus (iii) implies (i) and (ii). Let us show that (i) implies (iii). By~\eqref{eq:support_decomposition} we may assume that $\cF$ has strict support $Z$ for a closed subscheme $Z\subset Y$. Condition (i) implies that $Z\cap Y'\ne\emptyset$. Thus by the definition of intermediate extension and Proposition~\ref{pr:Saito3.21} we can write $\cF=j_{!*}\cF'$ for $\cF'\in\MH(Y',n)$. It remains to use Lemma~\ref{lm:j*composition}(i). We also see from Proposition~\ref{pr:Saito3.21} that there is a unique isomorphism $j_{!*}j^*\cF\to\cF$ whose restriction to $Y'$ is the identity. Similarly we show that (ii) implies (iii).
\end{proof}

There is a related statement for mixed Hodge modules.
\begin{lemma}\label{lm:EqualTo!*Mixed} Let $j\colon Y'\to Y$ be an open embedding and let $\cF\in\MHM(Y)$ be such that $j^*\cF$ is pure of weight $n$ and $\Rat(\cF)$ has no subobjects or quotient objects supported on $Y-Y'$. Then $\cF\in\MH(Y,n)$ and there is a unique isomorphism $j_{!*}j^*\cF\to\cF$ whose restriction to $Y'$ is the identity.
\end{lemma}
\begin{proof}
  Considering the weight filtration on $\cF$, we see that if $\cF$ is not pure of weight $n$, then there is either a submodule $\cF'$ whose weights are less than $n$, or a quotient module $\cF'$ whose weights are greater than $n$. Since the weight filtration is Zariski local, we see that $j^*\cF'=0$. Thus $\Rat(\cF')$ is supported on $Y-Y'$, which is a contradiction. We see that $\cF$ is pure of weight $n$ and we can apply Lemma~\ref{lm:EqualTo!*}.
\end{proof}

The next lemma follows from the definitions.
\begin{lemma}\label{lm:IntermExtRat}
$j_{!*}$ is compatible with $\Rat$ and the intermediate extension for perverse sheaves.
\end{lemma}

\subsection{Direct image for ind-proper morphisms}\label{ax:pushw}
We call a morphism $f\colon Y_1\to Y_2$ of ind-schemes \emph{ind-proper\/} if for any closed subscheme $Z\subset Y_2$ the fibered product $Y_1\times_{Y_2}Z$ is the inductive limit of proper $Z$-schemes (the morphisms being closed embeddings).

\begin{remark}
    It is not clear whether this definition is Zariski local over the base. However, the definition suffices for our purposes.
\end{remark}

The direct image of complexes of mixed Hodge modules on schemes is defined in~\cite{SaitoMixed}. If the morphism is proper, a version of decomposition theorem ensures that the direct image of a pure complex remains pure. We extend this to the ind-proper morphisms of ind-schemes in the proposition below.

\begin{proposition}\label{pr:DirectImage}
(i) Let $f\colon Y_1\to Y_2$ be an ind-proper morphism of ind-separated ind-schemes of ind-finite type over $\C$. Then there is a canonical exact functor
\[
    f_*\colon D^b\MH(Y_1,n)\to D^b\MH(Y_2,n)
\]
compatible with $\Rat$ and direct image for perverse sheaves in the sense that for any $\cF\in D^b\MH(Y_1,n)$ we have a canonical isomorphism
\[
    f_*\Rat(\cF)=\Rat(f_*\cF).
\]

(ii) Assume that $f$ is a closed embedding. Then $f_*$ is the closed embedding functor from Proposition~\ref{pr:KashDb}(ii).

(iii) Let $g\colon Y_2\to Y_3$ be another ind-proper morphism. Then we have a canonical isomorphism $(g\circ f)_*=g_*\circ f_*$.
\end{proposition}

\begin{proof}
Assume first that all our ind-schemes are, in fact, reduced schemes. By~\cite[Thm.~4.3]{SaitoMixed} we have the functors $f_*,f_!\colon D^b\MHM(Y_1)\to D^b\MHM(Y_2)$. Since $f$ is proper, we have $f_*=f_!$ by~\cite[(4.3.3)]{SaitoMixed}. Since $f_*$ does not decrease weights, and $f_!$ does not increase weights (see~\cite[Sect.~14.1.1]{PetersSteenbrinck}), $f_*$ actually preserves weights. Thus $f_*$ takes $D^b\MH(Y_1,n)$ into $D^b\MH(Y_2,n)$.

The other properties follow from the respective properties of the direct image of complexes of mixed Hodge modules.

If $Y_1$ and $Y_2$ are not necessarily reduced schemes, consider $f_{red}\colon Y_1^{red}\to Y_2^{red}$. We set $f_*:=(f_{red})_*$. The required properties follow immediately.

Finally, let $Y_i=\lim Y_i^{(l)}$ be ind-schemes of finite type. Every object of the category $D^b\MH(Y_1,n)$ can be written as $\iota^{(l)}_{1,*}\cF$, where $\cF$ is an object of $D^b\MH(Y_1^{(l)},n)$ and $\iota_1^{(l)}\colon Y_1^{(l)}\to Y_1$ is the embedding. Note that for large enough $l'$, the composition $f\circ\iota_1^{(l)}$ factors as $\iota_2^{(l')}\circ h$, where $\iota_2^{(l')}\colon Y_2^{(l')}\to Y_2$ is the embedding and $h\colon Y_1^{(l)}\to Y_2^{(l')}$ is proper. We define $f_*(\iota^{(l)}_{1,*}\cF):=\iota_{2,*}^{(l')}h_*\cF$. We define $f_*$ on morphisms similarly. We leave it to the reader to check that the functor is well-defined and to derive (i)--(iii) from the case of schemes.
\end{proof}

Note that $H^i(f_*\cF)=0$ for $|i|>\dim\supp\cF$.

\subsection{Duality} In~\cite{SaitoMixed} Saito defines the duality $D\colon D^b\MHM(Y)^{op}\to D^b\MHM(Y)$. The duality interchanges the direct images $f_*$ and $f_!$ (see~\cite[(4.3.5)]{SaitoMixed}). In particular, the duality commutes with $f_*=f_!$ if $f$ is proper. Applying this in the case, when $f$ is a closed embedding, we see that the duality is compatible with closed embeddings. This allows us to extend the duality to the case of ind-schemes. Note that the duality functor is compatible with $\Rat$ and the Verdier duality for perverse sheaves.

\begin{proposition}\label{pr:duality}
\stepzero\noindstep\label{duality:i} The duality functor takes $D^b\MH(Y,n)^{op}$ to $D^b\MH(Y,-n)$ and $\MH(Y,n)^{op}$ to $\MH(Y,-n)$. Also, the duality is compatible with the duality for variations of Hodge structures and Proposition~\ref{pr:Saito3.21}.\\
\noindstep\label{duality:ii}
The duality commutes with intermediate extensions and direct images for ind-proper morphisms.
\end{proposition}
\begin{proof}
It is enough to consider the case, when $Y$ is a reduced scheme. It follows, from~\cite[Prop~14.30(c)]{PetersSteenbrinck} and~\eqref{eq:support_decomposition} that the duality takes $\MH(Y,n)^{op}$ to $\MH(Y,-n)$. Thus it takes $D^b\MH(Y,n)^{op}$ to $D^b\MH(Y,-n)$.


Let us show that the duality commutes with the intermediate extensions. Using Kashiwara's equivalence, we reduce to the case of an open embedding $j\colon U\to Y$. Then for $\cF\in\MH(Y,n)$ the intermediate extension $j_{!*}\cF$ has no quotient objects supported on $Y-U$ so $Dj_{!*}\cF$ has no subobjects supported on $Y-U$. Next, the duality is compatible with restrictions to open subsets, 
so we have $j^*Dj_{!*}\cF=D\cF$ (we used Lemma~\ref{lm:j*composition}(i)) and the statement follows from Lemma~\ref{lm:EqualTo!*}.

The compatibility with the duality for VHS follows from the definitions. The compatibility with direct images has been already explained.
\end{proof}

\subsection{Exterior product}\label{ax:boxtimes} In~\cite{SaitoMixed}, Saito defines the exterior product of mixed Hodge modules
\[
    \boxtimes\colon\MHM(Y_1)\times\MHM(Y_2)\to\MHM(Y_1\times Y_2)
\]
satisfying the usual compatibilities. In particular, the exterior product is compatible with Kashiwara's equivalence, so we can extend it to the case of ind-schemes.

\begin{proposition}\label{pr:boxtimes}
(i) The exterior product is compatible with $Rat$. \\
(ii)The exterior product is compatible with the exterior product of VHS and Proposition~\ref{pr:Saito3.21}.\\
(iii) Let $j_1\colon Y_1'\to Y_1$ and $j_2\colon Y_2'\to Y_2$ be locally closed embeddings and assume that $\cF_1$ and $\cF_2$ are pure polarizable Hodge modules on $Y'_1$ and $Y'_2$ respectively. Then
\[
    (j_1\boxtimes j_2)_{!*}(\cF_1\boxtimes\cF_2)=(j_1)_{!*}\cF_1\boxtimes(j_2)_{!*}(\cF_2).
\]
(iv) The exterior product restricts to a functor
\[
    \boxtimes\colon\MH(Y_1,n_1)\times\MH(Y_2,n_2)\to\MH(Y_1\times Y_2,n_1+n_2).
\]
\end{proposition}
\begin{proof} The first two statements follow from the definitions. For (iii), using the properties of perverse sheaves we show that
$\Rat((j_1)_{!*}\cF_1\boxtimes(j_2)_{!*}\cF_2)$ has no subsheaves supported on
\[
    Y_1\times Y_2-Y'_1\times Y'_2=((Y_1-Y'_1)\times Y_2)\cup(Y_1\times(Y_2-Y'_2)).
\]
Now (iii) follows from Lemma~\ref{lm:EqualTo!*Mixed}.

It remains to prove (iv). It follows from (ii), (iii), and Proposition~\ref{pr:Saito3.21} that $\MH_{Z_1}(n_1)\boxtimes\MH_{Z_2}(n_2)\subset\MH_{Z_1\times Z_2}(n_1+n_2)$, whenever $Z_1$ and $Z_2$ are integral schemes. Finally, (iv) follows from~\eqref{eq:support_decomposition}.
\end{proof}

\subsection{Tate twists}\label{sect:TateTwists}
For the mixed Hodge modules on reduced separated schemes one has Tate twists $\cF\to\cF(k)$, where $k\in\Z$. They are auto-equivalences of $\MHM(X)$ sending $\MH(X,n)$ to $\MH(X,n-2k)$ (see, e.g.,~\cite[XIV--17, Prop.~14.30(b)]{PetersSteenbrinck}). The Tate twists are compatible with Kashiwara's equivalence, so they can be extended to ind-schemes. Note that the Tate twists are also compatible with $\Rat$.

\subsection{Inverse images}\label{sect:InvImage}
In~\cite[Sect.~4.4]{SaitoMixed} Saito defines the inverse images of mixed Hodge modules. That is, if $f\colon Y_1\to Y_2$ is a morphism of reduced separated schemes of finite type, then Saito defines functors $f^!,f^*\colon D^b\MHM(Y_2)\to D^b\MHM(Y_1)$ as right (resp.~left) adjoint to $f_!$ (resp.~$f_*$).

These functors satisfy the usual properties, in particular, the base change~\cite[(4.4.3)]{SaitoMixed} and compatibility with $\boxtimes$. This allows us to generalize these functors to schematic morphisms of ind-schemes. The functors $f^!$ and $f^*$ are also compatible with $\Rat$ and the corresponding functors for perverse sheaves.

We will need a description of these inverse images in the case, when $f$ is a smooth morphism of pure relative dimension. It will be convenient for us to use the shifted inverse image $f^\dagger:=f^![-d](-d)$, where $d$ is the relative dimension of $f$ (this is also equal to $f^*[d]$ as follows from the definition in~\cite[Sect.~2.17]{SaitoMixed}). Note that a similar functor for perverse sheaves takes perverse sheaves to perverse sheaves, so, by faithfulness of $\Rat$, $f^\dagger$ takes mixed Hodge modules to mixed Hodge modules.

\begin{proposition}\label{pr:InvImage}
Let $f\colon Y_1\to Y_2$ be a smooth schematic morphism of relative dimension $d$. Then\\
\stepzero
\noindstep\label{InvImage:iv} If $f$ is an open embedding, then $f^\dagger$ is the restriction functor.\\
\noindstep\label{InVImage:VHS} If $Y_2$ is a smooth scheme, then $f^\dagger$ is compatible with pullback of filtered $\EuD$-modules with rational structure and with Proposition~\ref{pr:Saito3.21}.\\
\noindstep\label{InvImage:0} If $Y_1$ and $Y_2$ are integral schemes, then $f^\dagger$ takes $\MH_{Y_2}(n)$ to $\MH_{Y_1}(n+d)$.\\
\noindstep\label{InvImage:i} $f^\dagger$ takes $\MH(Y_2,n)$ to $\MH(Y_1,n+d)$ and $D^b\MH(Y_2,n)$ to $D^b\MH(Y_1,n+d)$.
\end{proposition}
\begin{proof}
    Parts~\eqref{InvImage:iv} and~\eqref{InVImage:VHS} follow from the definition~\cite[(2.17.6)]{SaitoMixed} (see also~\cite[Sect.~3.5]{SaitoPolarizable}).

    Let us prove~\eqref{InvImage:0}. If $Y_2$ is smooth, the statement follows from~\cite[Lemma~2.25]{SaitoMixed} because being smooth $f$ is non-characteristic by~\cite[Lemme~3.5.2]{SaitoPolarizable}. In the general case, let $\cF\in\MH_{Y_2}(n)$, let $U_2$ be a non-empty smooth open subset of $Y_2$, and let $\cF':=\cF|_{U_2}$. Let $g$ be the restriction of $f$ to $U_1:=f^{-1}(U_2)$. By the smooth case, $g^\dagger\cF'\in\MH(U_1,n+d)$. Let $j\colon U_1\to Y_1$ be the open embedding. It is enough to show that $f^\dagger\cF=j_{!*}g^\dagger\cF'$, which follows from Lemma~\ref{lm:EqualTo!*Mixed} and the fact that smooth inverse images of perverse sheaves commute with intermediate extensions. Finally,~\eqref{InvImage:i} follows from part~\eqref{InvImage:0}, \eqref{eq:support_decomposition}, and Proposition~\ref{pr:KashDb}.
\end{proof}

We need a version of base change for intermediate extensions.
\begin{lemma}\label{lm:basechange}
Consider a Cartesian diagram of ind-schemes
\[
\begin{CD}
Y_1 @>j'>> Y_2\\
@Vf'VV @VfVV\\
Z_1 @>j>> Z_2
\end{CD}
\]
where $j$ is a locally closed embedding and $f$ is smooth and schematic. Then for all $\cF\in\MH(Z_1,n)$ we have functorially in $\cF$
\[
    f^\dagger j_{!*}\cF=j'_{!*}(f')^\dagger\cF.
\]
\end{lemma}
\begin{proof}
The proof is similar to that of Proposition~\ref{pr:InvImage}\eqref{InvImage:0} and is left to the reader.
\end{proof}

\begin{proposition}\label{pr:SmoothDescent}
Hodge modules satisfy descent for smooth schematic surjective morphisms.
\end{proposition}
\begin{proof}
Let $f\colon Y_1\to Y_2$ be a smooth surjective schematic morphism.

We first prove the descent for morphisms. Let $\cF,\cF'\in\MH(Y_2,n)$ and let $\phi\colon f^\dagger\cF\to f^\dagger\cF'$ be a morphism. We need to show that if the two pullbacks of $\phi$ to $Y_1\times_{Y_2}Y_1$ coincide, then there is a unique morphism $\psi\colon\cF\to\cF'$ such that $f^\dagger\psi=\phi$.

First of all, we may assume that $Y_2$ is a quasi-projective reduced scheme. Note that the natural map $\Hom(\cF,\cF')\to\Hom(f^\dagger\cF,f^\dagger\cF')$ is injective (compose it with $\Rat$). It follows that the statement is Zariski local over $Y_2$. Thus, we can assume that $Y_2$ is embeddable into a smooth scheme, and, in turn, that $Y_2$ is a smooth scheme. It remains to note that the category of Hodge modules on a smooth scheme is a full subcategory of the category of filtered $\EuD$-modules with rational structures and we have smooth descent for the latter category.

Now assume that $\cF\in\MH(Y_1,n)$ and suppose we are given an isomorphism between the two pullbacks of $\cF$ to $Y_1\times_{Y_2}Y_1$  satisfying the cocycle condition. We need to show that $\cF\simeq f^\dagger(\cF')$ for some $\cF'\in\MH(Y_2,n-d)$, where $d$ is the relative dimension of $f$. Since Hodge modules satisfy Zariski descent, we my replace $Y_1$ with an affine cover, so we assume that $f$ is affine. We can compactify $f$ to a projective morphism $\overline Y_1\to Y_2$. Using resolution of singularities we may assume that $\overline Y_1-Y_1$ is a locally principal divisor. If $j\colon Y_1\to\overline Y_1$ is the open immersion, then $j_{!*}\cF$ is an extension of $\cF$ to $\overline Y_1$. Next, $\Rat(\cF)$ descends to $Y_2$. This follows from the faithfulness of $\Rat$ and the fact that perverse sheaves satisfy smooth descent. Thus we can apply~\cite[Lemma~2.27]{SaitoMixed} to conclude that there is a mixed Hodge module $\cF'$ such that $\cF\simeq f^\dagger(\cF')$. It remains to show that $\cF'$ is pure of weight $n-d$. But this follows from~\cite[Lemma~2.25]{SaitoMixed}.
\end{proof}

\subsection{Equivariant Hodge modules on ind-schemes}\label{sect:EquivHodge} Assume that $Y$ is a separated scheme of finite type over $\C$ acted upon by a group $H$ of finite type over $\C$. Consider two smooth morphisms: the projection $p\colon H\times Y\to Y$ and the action $a\colon H\times Y\to Y$. An object of $\MH^H(Y,n)$ is a pair $(\cF,s)$, where $\cF\in\MH(Y,n)$, the equivariance morphism $s\colon p^\dagger\cF\xrightarrow{\simeq}a^\dagger\cF$ is an isomorphism, and $s$ satisfies the cocycle condition (note that by Proposition~\ref{pr:InvImage}\eqref{InvImage:i} $p^\dagger\cF$ and $a^\dagger\cF$ are pure Hodge modules). The morphisms in this category are morphisms of Hodge modules compatible with the equivariance isomorphisms.

\begin{lemma}\label{lm:ActsThroughQuotient}
Assume that $H\to H'$ is a surjective homomorphism of complex algebraic groups with connected kernel. If $H$ acts on $Y$ through its quotient $H'$, then we have a canonical equivalence of the categories $\MH^H(Y,n)=\MH^{H'}(Y,n)$.
\end{lemma}
\begin{proof}
Set $Z=H\times Y$, $Z'=H'\times Y$ and let $f\colon Z\to Z'$ be the projection. Let $\cF\in\MH(Y,n)$ and let $s$ be an equivariance isomorphism between the two pullbacks of $\cF$ to $Z$. We need to show that $s$ descends along $f$. Since a similar statement is known for perverse sheaves, 
we see that $\Rat(s)$ descends along $f$, which means that the two pullbacks of $\Rat(s)$ to $Z\times_{Z'}Z$ are equal. Since $\Rat$ is faithful, this implies that the two pullbacks of $s$ to $Z\times_{Z'}Z$ are also equal, and the statement follows from the descent (Proposition~\ref{pr:SmoothDescent}).
\end{proof}

Assume that $H=\lim\limits_{\infty\leftarrow m}H^{(m)}$ is a connected pro-algebraic complex group acting on an ind-scheme $Y$. We will assume that the morphisms $H^{(m)}\to H^{(m')}$ are surjective with connected kernels. Assume that $H$ acts \emph{nicely}, 
that is, every closed subscheme $Z\subset Y$ is contained in an $H$-invariant closed subscheme $Z'$ on which $H$ acts via one of its quotients $H_{(m)}$ (see~\cite[A.3]{GaitsgoryCentral}). We keep these assumptions on $H$ and the action through the end of Section~\ref{sect:EquivHodge}. We are going to define the category $\MH^H(Y,n)$ of pure polarizable $H$-equivariant Hodge modules of weight $n$. We can write $Y=\lim\limits_{l\to\infty} Y^{(l)}$, where $Y^{(l)}$ is $H$-invariant and $H$ acts through $H_{(m_l)}$ for a certain $m_l$. We may assume that $m_l$ is a non-decreasing sequence.

For $l<l'$ we have a functor
\[
   F_{l,l'}\colon\MH^{H_{(m_l)}}(Y^{(l)},n)=\MH^{H_{(m_{l'})}}(Y^{(l)},n)\to\MH^{H_{(m_{l'})}}(Y^{(l')},n),
\]
where the equality follows from Lemma~\ref{lm:ActsThroughQuotient} and the direct image functor is defined with a help of base change. It is easy to check that we have a coherent system of isomorphisms of functors $F_{l',l''}\circ F_{l,l'}\simeq F_{l,l''}$ defined whenever $l<l'<l''$. Set
\[
   \MH^H(Y,n):=\lim_{l\to\infty}\MH^{H_{(m_l)}}(Y^{(l)},n).
\]
One checks that this category does not depend on the choices up to a canonical equivalence.

\begin{proposition}\label{pr:ForgetEquiv}
\stepzero\noindstep\label{ForgetEquiv:i} The forgetful functor $\For\colon\MH^H(Y,n)\to\MH(Y,n)$ is a fully faithful embedding.\\
\noindstep\label{ForgetEquiv:ii} $\MH^H(Y,n)$ is a semisimple abelian $\Q$-linear category.
\end{proposition}
\begin{proof}
    \eqref{ForgetEquiv:i} We may assume that $Y$ is a reduced scheme. It is clear that $\For$ is faithful. Let us show that $\For$ is full. We take $\cF,\cF'\in\MH^H(Y,n)$ and let $\phi\colon\For(\cF)\to \For(\cF')$ be a morphism. We need to show that it respects the $H$-equivariance structure. A similar statement is well known for perverse sheaves. 
    Thus $\Rat(\phi)$ respects the equivariant structures. It follows from the faithfulness of $\Rat$ that $\phi$ respects the equivariant structure as well.

    \eqref{ForgetEquiv:ii} It is easy to see that $\MH^H(Y,n)$ is a $\Q$-linear abelian category. Now it follows from part~\eqref{ForgetEquiv:i} and Proposition~\ref{pr:HM_basic_properties}\eqref{basic:i} that it is semisimple: indeed, a full abelian subcategory of a semisimple abelian category is necessarily semisimple.
\end{proof}

We use this proposition to identify the category $\MH^H(Y,n)$ with a full subcategory of $\MH(Y,n)$. Thus, we may view equivariance as a property rather than a structure. It is clear that $\Rat$ takes $\MH^H(Y,n)$ to $\Perv_\Q^H(Y,n)$. The proof of the following easy lemma is left to the reader.
We let $D^b\MH^H(Y,n)$ be the full subcategory of $D^b\MH(Y,n)$ whose objects are complexes with equivariant cohomology groups.

\begin{lemma}\label{lm:EquivPushFrwrd} Let $H$ act nicely on ind-schemes $Y_1$ and $Y_2$ and let $f\colon Y_1\to Y_2$ be an equivariant morphism, then:\\
(i) If $f$ is ind-proper, then $f_*(D^b\MH^H(Y_1,n))\subset D^b\MH^H(Y_2,n)$.\\
(ii) If $f$ is smooth of relative dimension $d$, then $f^\dagger(\MH^H(Y_2,n))\subset\MH^H(Y_1,n+d)$.\\
(iii) If $f$ is a locally closed embedding, then $f_{!*}(\MH^H(Y_1,n))\subset\MH^H(Y_2,n)$.
\end{lemma}

We have the following consequence of smooth descent (Proposition~\ref{pr:SmoothDescent}).

\begin{proposition}\label{pr:TorsorDescent}
  Let $f\colon Y_1\to Y_2$ be an $H$-torsor, where $H$ is an algebraic group acting nicely on $Y_1$. Then the functor $f^\dagger$ induces an equivalence of the categories $\MH(Y_2,n)\xrightarrow{\simeq}\MH^H(Y_1,n+d)$, where $d=\dim H$. In particular, if $H$ is connected, then $f^\dagger$ identifies $\MH(Y_2,n)$ with a full subcategory of $\MH(Y_1,n+d)$.
\end{proposition}

Since $H$ is assumed to act nicely on $Y$, every orbit $\cO\to Y$ is a locally closed subscheme. Let $j_\cO\colon\cO\to Y$ be the locally closed embedding and let $a_\cO\colon\cO\to pt$ be the projection to the point. Let $\cH$ be a pure polarizable Hodge structure. Set $\IC_\cO(\cH):=(j_\cO)_{!*}a_\cO^\dagger\cH$.

We will say that $H$ acts on $Y$ with \emph{ind-finitely many orbits\/} if we can write $Y=\lim_{l\to\infty}Y^{(l)}$, where $Y^{(l)}$ is a closed $H$-invariant subscheme on which $H$ acts with finitely many orbits. The following proposition give full description of the abelian semisimple category $\MH^H(Y,n)$ in the case when $H$ acts with ind-finitely many orbits.

\begin{proposition}\label{pr:IndFinitelyMany}
Let $H$ act nicely and with ind-finitely many orbits on an ind-scheme $Y$.

(i) The forgetful functor $\MH^H(Y,n)\to\MH(Y,n)$ induces an equivalence between $\MH^H(Y,n)$ and the strictly full subcategory in $\MH(Y,n)$ whose simple objects are isomorphic to $\IC_\cO(\cH)$, where $\cO$ ranges over the set of orbits and $\cH$ ranges over irreducible polarizable Hodge structures of weight $n-\dim\cO$.

(ii) For Hodge structures $\cH$ and $\cH'$ and an orbit $\cO$ we have a canonical isomorphism
\[
    \Hom(\IC_\cO(\cH),\IC_{\cO}(\cH')=\Hom(\cH,\cH').
\]
\end{proposition}
\begin{proof}
(i) Let $\cH$ be an irreducible polarizable Hodge structure, then it follows from Lemma~\ref{lm:EquivPushFrwrd} that $\IC_\cO(\cH)$ has an equivariant structure. Assume that it can be written as the direct sum $\cF_1\oplus\cF_2$. Then $\cF_1$ and $\cF_2$ necessarily have strict support $\overline\cO$. Thus, over some open subset of $\cO$ both $\cF_1$ and $\cF_2$ are variations of Hodge structures. This, however, contradicts irreducibility of $\cH$. Thus $\IC_\cO(\cH)$ is a simple object.

Conversely, assume that $\cF$ is a simple object of $\MH^H(Y,n)$. Then it is also a simple object of $\MH(Y,n)$ (as otherwise we would have a non-trivial idempotent endomorphism).  By~\eqref{eq:support_decomposition} $Z=\supp(\cF)$ is an integral subscheme and we can write $\cF=i_*\cF'$, where $i\colon Z\to Y$ is the closed embedding. Since $\Rat(\cF)$ is $H$-equivariant, we see that $Z$ is $H$-invariant and that $\cF'\in\MH^H(Z,n)$. Since the action has ind-finitely many orbits, some orbit $\cO$ is open in $Z$. By Proposition~\ref{pr:Saito3.21}, generically on $Z$, $\cF'$ is a polarizable variation of Hodge structures. Let $W$ be the largest open subset of $Z$ such that $\cF'|_W$ is a variation of Hodge structures. It is easy to see that $W$ is $H$-invariant. 
We see that $\cO\subset W$. Thus, $\cV:=\cF'|_\cO$ is an $H$-equivariant variation of Hodge structures. Since $H$ acts transitively on $\cO$, it is easy to see that~$\cV$ is isomorphic to~$a_\cO^\dagger\cH$ for a Hodge structure $\cH$. It is clear that $\cH$ is polarizable and has weight $n-\dim\cO$. Now, by Proposition~\ref{pr:Saito3.21} we have
\[
    \cF\simeq(j_\cO)_{!*}\cV\simeq\IC_\cO(\cH).
\]
Clearly, $\cH$ is irreducible.

(ii) It follows from the definitions that $\IC_\cO(\cH)$ and $\IC_{\cO}(\cH')$ have strict support equal to the closure of $\cO$. Now, using Kashiwara's equivalence and Proposition~\ref{pr:Saito3.21}(ii) we get an isomorphism
\[
    \Hom(\IC_\cO(\cH),\IC_{\cO}(\cH')=\Hom(a_\cO^\dagger\cH,a_\cO^\dagger\cH')=\Hom(\cH,\cH'),
\]
where the last isomorphism follows, for example, from Proposition~\ref{pr:InvImage}\eqref{InVImage:VHS} and a similar statement for constant VHS.
\end{proof}

\section{Tannakian structure on $\MH^{G_\cO}(\Gr_G)$}\label{sect:tannakian}
In this section we equip $\MH^{G_\cO}(\Gr_G)$ with a structure of neutral Tannakian category.

\subsection{$G_\cO$-equivariant Hodge modules on $\Gr_G$}\label{sect:irr}
Put $\cK:=\C((t))=\C[[t]][t^{-1}]$. Let $G_\cK$ denote the loop group representing the functor $R\mapsto G\bigl(R((t))\bigr)$ from the category of $\C$-algebras to the category of sets. Then the affine Grassmannian is canonically isomorphic to the quotient $G_\cK/G_\cO$.

Fix a Borel subgroup $B\subset G$ and a split maximal torus $T\subset B$. Let $X_*:=\Hom(\gm\C,T)$ be the co-character lattice and let $X_+\subset X_*$ be the semigroup of $B$-dominant co-characters. Every $\lambda\in X_*$ gives a point $t^\lambda\in G_\cK(\C)$ defined by post-composing $\lambda$ with $T\hookrightarrow G$ and pre-composing with $\Spec\cK\to\gm\C$. Consider the $G_\cO$-orbit $\Gr^\lambda:=G_\cO\cdot(t^\lambda\cdot G_\cO)\subset\Gr_G$. It is well-known that every $G_\cO$-orbit in $\Gr_G$ is of the form $\Gr^\lambda$ and $\Gr^\lambda=\Gr^{\lambda'}$ if and only if $\lambda$ and $\lambda'$ are in the same orbit under the action of the Weyl group of $G$ on $X_*$. Thus we have a stratification
\[
    \Gr_G=\bigcup_{\lambda\in X_+}\Gr^\lambda.
\]
For the proofs we refer the reader to~\cite{BruhatTits} and~\cite{AlperHalpernHeinlotch_Decomposition}.

Recall that we have a presentation $\Gr_G=\lim_{l\to\infty}\Gr_G^{(l)}$, where $\Gr_G^{(l)}$ are $G_\cO$-invariant projective schemes (cf.~\eqref{eq:GrIndSch}). Note that for $m\gg0$ the group $G_\cO$ acts on $\Gr_G^{(l)}$ through a finite dimensional quotient $G_{(m)}$, where $G_{(m)}$ represents the functor $R\mapsto G(R[t]/t^m)$. In particular, the group $G_\cO$ acts nicely on $\Gr_G$ and we can consider the category $\MH^{G_\cO}(\Gr_G,n)$. Recall that $\MH^{G_\cO}(\Gr_G)=\bigoplus_{n\in\Z}\MH^{G_\cO}(\Gr_G,n)$.

Take $\lambda\in X_+$ and let $\cH$ be a pure polarizable Hodge structure of weight $n-\dim\Gr^\lambda$. Set $\IC_\lambda(\cH):=\IC_{\Gr^\lambda}(\cH)$ so that $\IC_\lambda(\cH)\in\MH^{G_\cO}(\Gr_G,n)$.

Combining Proposition~\ref{pr:ForgetEquiv} and Proposition~\ref{pr:IndFinitelyMany} we get the description of $\MH^{G_\cO}(\Gr_G)$ up to equivalence:

\begin{proposition}\label{pr:equiv}
(i) The forgetful functor $\MH^{G_\cO}(\Gr_G)\to\MH(\Gr_G)$ induces an equivalence between $\MH^{G_\cO}(\Gr_G)$ and a full abelian semisimple subcategory in $\MH(\Gr_G)$ whose simple objects are isomorphic to $\IC_\lambda(\cH)$, where $\lambda$ ranges over $X_+$ and $\cH$ ranges over irreducible Hodge structures.

(ii) For Hodge structures $\cH$ and $\cH'$ we have a canonical isomorphism
\[
    \Hom(\IC_\lambda(\cH),\IC_\lambda(\cH'))=\Hom(\cH,\cH').
\]
\end{proposition}

We often use this proposition tacitly to identify the two categories.

\subsection{Convolution of Hodge modules}\label{sect:conv} We want to define a convolution on the category $\MH^{G_\cO}(\Gr_G)$. Consider the convolution Grassmannian
\[
    \Gr_{conv}:=G_\cK\times^{G_\cO}\Gr_G.
\]
This is a fibration over $\Gr_G$ with the fiber isomorphic to $\Gr_G$. In particular, $\Gr_{conv}$ is an ind-proper ind-scheme. Let $\cF_1,\cF_2\in\MH^{G_\cO}(\Gr_G)$. We first need to define the twisted exterior product $\cF_1\twext\cF_2$ on $\Gr_{conv}$. Morally, the definition is as follows. Let $\tilde\cF_1$ be the pullback of $\cF_1$ to $G_\cK$. Then the $G_\cO$-equivariance shows that $\tilde\cF_1\boxtimes\cF_2$ descends from $G_\cK\times\Gr_G$ to the Hodge module $\cF_1\twext\cF_2$ on $\Gr_{conv}$. We prefer to avoid using Hodge modules on $G_\cK$ because $G_\cK$ is not of ind-finite
type.

We proceed as follows. Let $l$ be large enough, then we may view $\cF_2$ as an object of $\MH^{G_\cO}(\Gr_G^{(l)})$. Let $G_\cO$ act on $\Gr_G^{(l)}$ through its quotient $G_{(m)}$. Thus, we may view $\cF_2$ as an object of $\MH^{G_{(m)}}(\Gr_G^{(l)})$ (by Lemma~\ref{lm:ActsThroughQuotient}).

Denote the kernel of $G_\cO\to G_{(m)}$ by $G^{(m)}$ (this is the $m$-th congruence subgroup) and consider the ind-scheme of ind-finite type $G_\cK/G^{(m)}$; the morphism $p\colon G_\cK/G^{(m)}\to\Gr_G$ is a $G_{(m)}$-torsor.  Consider the closed embedding.
\begin{equation}\label{eq:j}
        i_{m,l}\colon(G_\cK/G^{(m)})\times^{G_{(m)}}\Gr_G^{(l)}=
        G_\cK\times^{G_\cO}\Gr_G^{(l)}\to\Gr_{conv}.
\end{equation}
The Hodge module $p^\dagger\cF_1\boxtimes\cF_2\in\MH((G_\cK/G^{(m)})\times\Gr_G^{(l)})$ is easily seen to be $G_{(m)}$-equivariant, so it descends to a Hodge module $\tilde\cF$ on $(G_\cK/G^{(m)})\times^{G_{(m)}}\Gr_G^{(l)}$ by Proposition~\ref{pr:TorsorDescent}. We set
\[
  \cF_1\twext\cF_2:=(i_{m,l})_*\tilde\cF.
\]

Let $m\colon\Gr_{conv}\to\Gr_G$ be the morphism, induced by the action of $G_\cK$ on $\Gr_G$. It is ind-proper because $\Gr_{conv}$ is ind-proper.
\begin{definition}\label{def:conv}
\[
    \cF_1\star\cF_2:=H^0(m_*(\cF_1\twext\cF_2)).
\]
\end{definition}

It is easy to see that the convolution is functorial.

\begin{lemma}\label{lm:RatMonoidal}
\stepzero\noindstep\label{rat_i} We have $H^j(m_*(\cF_1\twext\cF_2))=0$ for all $j\ne0$.

\noindstep\label{rat_ii} $\Rat$ is a monoidal functor.
\end{lemma}
\begin{proof}
Using the fact that $\Rat$ commutes with $m_*$ (Proposition~\ref{pr:DirectImage}), with $\boxtimes$ (Proposition~\ref{pr:boxtimes}(i)), and with smooth pullbacks (Proposition~\ref{pr:InvImage}), it is easy to see that
\[
    \Rat(m_*(\cF_1\twext\cF_2))=\Rat(\cF_1)\star\Rat(\cF_2),
\]
where the convolution on perverse sheaves is defined in~\cite[(4.2)]{MirkovicVilonen}. To prove (i), since $\Rat$ is conservative (Proposition~\ref{pr:DerRat}), we only need to show that $\Rat(m_*(\cF_1\twext\cF_2))$ is concentrated in cohomological degree zero. This follows from the above equality and~\cite[Prop~4.2]{MirkovicVilonen}. Now (ii) is also clear.
\end{proof}

\begin{lemma}
\[
    \cF_1\star\cF_2\in\MH^{G_\cO}(\Gr_G).
\]
\end{lemma}
\begin{proof}
Let $G_\cO$ act on $\Gr_{conv}$ via left multiplication on $G_\cK$. Let $G_\cO$ act on $G_\cK/G^{(m)}$ similarly. It is clear that~\eqref{eq:j} is $G_\cO$-equivariant. Lemma~\ref{lm:EquivPushFrwrd}(ii,iii) shows that $\cF_1\twext\cF_2$ is $G_\cO$-equivariant. Since $m$ is $G_\cO$-equivariant, the statement follows from Lemma~\ref{lm:EquivPushFrwrd}(i).
\end{proof}

We note that the convolution is compatible with weights, that is,
\[\MH^{G_\cO}(\Gr_G,n_1)\star \MH^{G_\cO}(\Gr_G,n_2)\subset \MH^{G_\cO}(\Gr_G,n_1+n_2).\]

One can define triple convolution along the same lines and use it to show that the convolution is associative up to a canonical isomorphism, that is, the pentagon identities are satisfied (one can also derive the pentagon identities from the faithfulness of $\Rat$).

\subsection{Fusion product and the commutativity constraint} Recall from \cite[Sect.~3.1.1]{GaitsgoryCentral} the following version of Beilinson--Drinfeld Grassmannian $\widetilde\Gr$. This is the ind-scheme classifying the data $(x,\cE,s)$, where $x\in\A^1$, $\cE$ is a $G$-torsor on $\A^1$, and $s$ is a trivialization of $\cE$ over $\A^1-x-0$. (For simplicity, we work with the curve $\A^1$, as it has a global coordinate but we could work with any curve). The ind-scheme $\widetilde\Gr$ is ind-separated and of ind-finite type over $\C$ so we can talk about Hodge modules on $\widetilde\Gr$. We have a morphism $\widetilde\Gr\to\A^1$ given by $(x,\cE,s)\mapsto x$. We use this morphism to view $\widetilde\Gr$ as a family over $\A^1$.

Let $X$ be a connected curve over $\C$ and let $x\in X$ be a smooth closed point of $X$. Recall that the modular interpretation of $\Gr_G$ identifies it with the moduli space of pairs $(\cE,s)$ where $\cE$ is a $G$-torsor over $X$, $s$ is a trivialization of $\cE$ on $X-x$. This identification depends only on a choice of a formal coordinate at $x$. Applying this to $X=\A^1$, $x=0$, we see that the 0-fiber of the family $\widetilde\Gr\to\A^1$ is identified with $\Gr_G$. The well-known factorization property gives a canonical isomorphism (cf.~\cite[Prop.~5]{GaitsgoryCentral}).
\[
    \widetilde\Gr|_{\A^1-0}\xrightarrow{=}
    \Gr_G\times(\A^1-0)\times\Gr_G\colon(x,\cE,s)\mapsto((\cE|_{\A^1-0},s),x,(\cE|_{\A^1-x},s)).
\]
Thus we have an open embedding
\[
    j\colon\Gr_G\times(\A^1-0)\times\Gr_G=\widetilde\Gr|_{\A^1-0}\to\widetilde\Gr
\]
and a closed embedding
\[
    i\colon\Gr_G=\widetilde\Gr|_0\to\widetilde\Gr.
\]

We make a convention. If $Y$ is an ind-scheme over $\A^1$, let $i$ be the closed embedding of the zero fiber, $j$ be the open embedding of the complement. Then for $\cF\in\MH(Y)$ we write
\[
    \cF|_0:=i^*\cF[-1],\quad \cF|_{\A^1-0}:=j^*\cF.
\]
Thus $\cF|_{\A^1-0}$ is again a direct sum of pure Hodge modules, while $\cF|_0$ may, in general, be a complex of mixed Hodge modules. We note that
\begin{equation}\label{eq:resA1}
    \Q_{\A_1}^{\mathrm{Hdg}}|_0=i^*a^*\Q^{\mathrm{Hdg}}=\Q^{\mathrm{Hdg}},
\end{equation}
where $a\colon\A^1\to pt$ is the projection.

Let $\cH_0$ be the unique Hodge structure on $\Q=\Q^1$ of weight zero. For a smooth scheme $Z$ we will denote the Hodge module $a^\dagger\cH_0$ by $\Q^{\mathrm{Hdg}}_Z$, where $a\colon Z\to pt$ is the projection. The following is essentially the Hodge module version of~\cite[Prop.~6(a)]{GaitsgoryCentral}.

\begin{proposition}\label{pr:Fusion} Assume that $\cF_1$ and $\cF_2$ are in $\MH^{G_\cO}(\Gr_G)$. Then we have a canonical isomorphism
\[
    \cF_1\star\cF_2=\bigl(j_{!*}(\cF_1\boxtimes\Q^{\mathrm{Hdg}}_{\A^1-0}\boxtimes\cF_2)\bigr)|_0.
\]
\end{proposition}
\begin{proof}
We introduce the convolution Grassmannian $\widetilde\Gr_{conv}$ classifying the data $(x,\cE_1,\cE_2,s_2,s_{21})$, where $x\in\A^1$, $\cE_i$ are $G$-torsors over $\A^1$, $s_2$ is a trivialization of $\cE_2$ over $\A^1-x$ (viewed as an isomorphism from the trivial torsor to $\cE_2$), $s_{21}$ is an isomorphism from $\cE_2$ to $\cE_1$ over $\A^1-x$. We view $\widetilde\Gr_{conv}$ as a family over $\A^1$. A modular description of $\Gr_{conv}$ identifies $\widetilde\Gr_{conv}|_0$ with $\Gr_{conv}$. More precisely, let $(g,(\cE,s))\in G_\cK\times\Gr_G$ be a representative of a point of $G_\cK\times^{G_\cO}\Gr_G=\Gr_{conv}$. Let $\cE'$ be obtained by gluing the trivial $G$-torsor over $\A^1-0$ with the trivial $G$-torsor on the formal disc centered at 0 via $g$. Let $s'$ be the canonical trivialization of $\cE'$ on $\A^1-0$. Then the embedding $\tilde\imath\colon\Gr_{conv}\to\widetilde\Gr_{conv}$ sends $(g,(\cE,s))$ to $(0,\cE,\cE',s',s)$.

We also have a factorization $\widetilde\Gr_{conv}|_{\A^1-0}=\Gr_G\times(\A^1-0)\times\Gr_G$ given by \[(x,\cE_1,\cE_2,s_2,s_{21})\mapsto((\cE_1|_{\A^1-0},s_{21}|_{\A^1-x-0}\circ s_2|_{\A^1-x-0}),x,(\cE_2,s_2)).\]
We have a morphism $\tilde m\colon\widetilde\Gr_{conv}\to\widetilde\Gr$ sending $(x,\cE_1,\cE_2,s_2,s_{21})$ to $(x,\cE_1,s_{21}|_{\A^1-x-0}\circ s_2|_{\A^1-x-0})$. It is not difficult to see that $\widetilde\Gr_{conv}$ is ind-proper over $\A^1$, so $\tilde m$ is ind-proper. We have a commutative diagram with Cartesian squares
\begin{equation}\label{eq:CD}
\begin{CD}
\Gr_{conv}@>\tilde\imath>>\widetilde\Gr_{conv}@<\tilde\jmath<<\Gr_G\times(\A^1-0)\times\Gr_G\\
@V m VV @V \tilde m VV @|\\
\Gr_G @>i>> \widetilde\Gr @<j<<\Gr_G\times(\A^1-0)\times\Gr_G\\
@VVV @V \pi VV @VVV\\
0 @>>> \A^1 @<<<\A^1-0.
\end{CD}
\end{equation}

Let $\cF'_1\in\MH(\Gr_G\times\A^1)$ and $\cF_2\in\MH(\Gr_G)$. We define the \emph{outer convolution\/} by
\[
    \cF'_1\star_o\cF_2:=\tilde m_*\tilde\jmath_{!*}((\cF'_1|_{\A^1-0})\boxtimes\cF_2)\in D^b\MH(\widetilde\Gr)
\]
(cf.~\cite[Def.~5]{GaitsgorySeminar16}).

\begin{lemma}\label{lm:OuterConv}
Let $\cF'_1\in\MH^{G_\cO}(\Gr_G\times\A^1)$ and $\cF_2\in\MH^{G_\cO}(\Gr_G)$. Then we have
\[
    \cF'_1\star_o\cF_2=j_{!*}((\cF'_1|_{\A^1-0})\boxtimes\cF_2).
\]
In particular, the outer convolution is concentrated in cohomological dimension zero.
\end{lemma}
\begin{proof} 
We claim that
\[
    \tilde m_*\tilde\jmath_{!*}((\cF'_1|_{\A^1-0})\boxtimes\cF_2)=j_{!*}j^*\tilde m_*\tilde\jmath_{!*}((\cF'_1|_{\A^1-0})\boxtimes\cF_2).
\]
By Lemma~\ref{lm:EqualTo!*}, it is enough to prove a similar statement for the underlying perverse sheaves, where it is equivalent to the statement that
\[
    \tilde m_*\tilde\jmath_{!*}\Rat((\cF'_1|_{\A^1-0})\boxtimes\cF_2)
\]
has no vanishing cycles and its nearby cycles have trivial monodromy. This follows from the fact that vanishing and nearby cycles commute with proper direct images (cf.~\cite[p.~7]{GaitsgorySeminar16}).

Thus by base change and Lemma~\ref{lm:j*composition}(i) we have
  \[
    \cF'_1\star_o\cF_2=j_{!*}j^*\tilde m_*\tilde\jmath_{!*}((\cF'_1|_{\A^1-0})\boxtimes\cF_2)=j_{!*}((\cF'_1|_{\A^1-0})\boxtimes\cF_2).
  \]
\end{proof}

Next, we have the following lemma.
\begin{lemma}\label{lm:TwextOuter}
  For $\cF_1,\cF_2\in\MH^{G_\cO}(\Gr_G)$ we have a canonical isomorphism
  \[
    \cF_1\twext\cF_2=(\tilde\jmath_{!*}(\cF_1\boxtimes\Q^{\mathrm{Hdg}}_{\A^1-0}\boxtimes\cF_2)|_0.
  \]
\end{lemma}
\begin{proof}
Let $l$ and $m$ be as in the definition of $\twext$ (Section~\ref{sect:conv}). Consider the closed sub-ind-scheme
$\widetilde{\Gr}_{conv}^{(l)}$ of $\widetilde{\Gr}_{conv}$ consisting of $(x,\cE_1,\cE_2,s_2,s_{21})$ such that $(\cE_2,s_2)\in\Gr_G^{(l)}$. Next, consider the ind-scheme $\widetilde{\Gr}_{conv}^{(l,m)}$ classifying $(x,\cE_1,\cE_2,s_2,s_{21},t_2)$, where $(x,\cE_1,\cE_2,s_2,s_{21})\in\widetilde{\Gr}_{conv}^{(l)}$, $t_2$ is a trivialization of $\cE_2$ on the $m$-th infinitesimal neighborhood of 0. Clearly, this is a $G_{(m)}$-torsor over $\widetilde{\Gr}_{conv}^{(l)}$. One checks that $\widetilde{\Gr}_{conv}^{(l,m)}$ factors as $(G_\cK/G^{(m)})\times\A^1\times\Gr_G^{(l)}$
and we have a diagram with cartesian squares
\[
\begin{CD}
G_\cK/G^{(m)}\times\Gr_G^{(l)}@>>>G_\cK/G^{(m)}\times\A^1\times\Gr_G^{(l)}@<<<
G_\cK/G^{(m)}\times(\A^1-0)\times\Gr_G^{(l)}\\
@V q VV @V \tilde p VV @VVV\\
G_\cK/G^{(m)}\times^{G_{(m)}}\Gr_G^{(l)}@>>>\widetilde\Gr_{conv}^{(l)}
@<<<\Gr_G\times(\A^1-0)\times\Gr_G^{(l)}\\
@VVV @VVV @VVV\\
\Gr_{conv}@>\tilde\imath>>\widetilde\Gr_{conv}@<\tilde\jmath<<\Gr_G\times(\A^1-0)\times\Gr_G\\
@VVV @VVV @VVV\\
0 @>>> \A^1 @<<<\A^1-0.
\end{CD}
\]
We use Kashiwara's equivalence to identify Hodge modules on the ind-schemes in the second row with Hodge modules on the corresponding ind-schemes on the third row. Then the base change Lemma~\ref{lm:basechange}, the compatibility of exterior products with inverse images and intermediate extensions (Proposition~\ref{pr:boxtimes}(iii)), and Lemma~\ref{lm:EqualTo!*} give
\[
    \tilde p^\dagger(\tilde\jmath_{!*}(\cF_1\boxtimes\Q^{\mathrm{Hdg}}_{\A^1-0}\boxtimes\cF_2))=\tilde\cF_1\boxtimes\Q^{\mathrm{Hdg}}_{\A^1}\boxtimes\cF_2,
\]
where $\tilde\cF_1$ is the pullback of $\cF_1$ to $G_\cK/G^{(m)}$. Using the compatibility of $\boxtimes$ with inverse images and~\eqref{eq:resA1}, we get
\[
    q^\dagger\left(\tilde\jmath_{!*}(\cF_1\boxtimes\Q^{\mathrm{Hdg}}_{\A^1-0}\boxtimes\cF_2)|_0\right)=\tilde\cF_1\boxtimes\cF_2.
\]
Since this Hodge module is $G_{(m)}$-equivariant, the statement follows from the descent (Proposition~\ref{pr:TorsorDescent}).
\end{proof}

Now we can prove Proposition~\ref{pr:Fusion}. Using Definition~\ref{def:conv}, Lemma~\ref{lm:TwextOuter}, base change, the definition of the outer product, Lemma~\ref{lm:OuterConv}, and compatibility of $\boxtimes$ with restrictions, we get
\begin{multline*}
    \cF_1\star\cF_2=m_*(\cF_1\twext\cF_2)=m_*((\tilde\jmath_{!*}(\cF_1\boxtimes\Q^{\mathrm{Hdg}}_{\A^1-0}\boxtimes\cF_2)|_0)=\\
    (\tilde m_*\tilde\jmath_{!*}(\cF_1\boxtimes\Q^{\mathrm{Hdg}}_{\A^1-0}\boxtimes\cF_2))|_0=
    ((\cF_1\boxtimes\Q^{\mathrm{Hdg}}_{\A^1})\star_o\cF_2)|_0=(j_{!*}(\cF_1\boxtimes\Q^{\mathrm{Hdg}}_{\A^1-0}\boxtimes\cF_2))|_0.
\end{multline*}
\end{proof}

Now we can define the commutativity constraint. Let $sw$ be the automorphism of $\Gr_G\times(\A^1-0)\times\Gr_G$ switching two copies of $\Gr_G$ and taking $x\in\A^1$ to $-x$. Let $\widetilde{sw}\colon\widetilde\Gr\to\widetilde\Gr$ be the automorphism given by $(x,\cE,s)\mapsto(-x,(sh_{-x})^*\cE,(sh_{-x})^*s)$, where $sh_x\colon\A^1\to\A^1$ is the shift by $-x$. It is easy to see that under the factorization isomorphism we have
\[
    sw=\widetilde{sw}|_{\A^1-0}.
\]
Thus by Proposition~\ref{pr:Fusion}
\begin{multline*}
    \cF_2\star\cF_1=
    \bigl(j_{!*}(\cF_2\boxtimes\Q^{\mathrm{Hdg}}_{\A^1-0}\boxtimes\cF_1)\bigr)\bigl|_0=
    \bigl(j_{!*}sw^*(\cF_1\boxtimes\Q^{\mathrm{Hdg}}_{\A^1-0}\boxtimes\cF_2)\bigr)\bigl|_0=\\
    \bigl(\widetilde{sw}^*j_{!*}(\cF_1\boxtimes\Q^{\mathrm{Hdg}}_{\A^1-0}\boxtimes\cF_2)\bigr)\bigl|_0=
    \bigl(j_{!*}(\cF_1\boxtimes\Q^{\mathrm{Hdg}}_{\A^1-0}\boxtimes\cF_2)\bigr)\bigl|_0=\cF_1\star\cF_2,
\end{multline*}
We used Lemma~\ref{lm:basechange} and the equation $\widetilde{sw}\circ i=i$. It is easy to see that we indeed get a commutativity constraint, that is, the isomorphism is involutive. This commutativity constraint is also easily seen to be compatible with $\Rat$ and the commutativity constraint from~\cite{MirkovicVilonen}. Thus the hexagon identities follow from faithfulness of $\Rat$. 

\subsection{Identity object}
\begin{lemma}
For $\lambda,\mu\in X_+$ and for polarizable pure Hodge structures $\cH$ and $\cH'$ let $l$ be such that $\Gr^\mu$ factors through $\Gr_G^{(l)}$. Let $\Gr^{\lambda,(m)}$ be the preimage of $\Gr^\lambda$ in $G_\cK/G^{(m)}$. Consider the locally closed embedding defined as the composition
\begin{equation}\label{eq:jj}
        j_{\lambda,\mu}:\Gr^{\lambda,(m)}\times^{G_\cO^{(m)}}\Gr^\mu\to(G_\cK/K^{(m)})\times^{G_\cO^{(m)}}\Gr_G^{(l)}=
        G_\cK\times^{G_\cO}\Gr_G^{(l)}\to\Gr_{conv}.
\end{equation}
Define the projection $a_{\lambda,\mu}:\Gr^{\lambda,(m)}\times^{G_\cO^{(m)}}\Gr^\mu\to pt$. Then we have a canonical isomorphism
\begin{equation}
    \IC_\lambda(\cH)\twext\IC_\mu(\cH'):=(j_{\lambda,m})_{!*}a_{\lambda,m}^\dagger(\cH\otimes\cH').
\end{equation}
\end{lemma}
\begin{proof}
  This follows from base change (Lemma~\ref{lm:basechange}), the properties of $\boxtimes$ (Proposition~\ref{pr:boxtimes}), and the descent for torsors (Proposition~\ref{pr:TorsorDescent}).
\end{proof}

\begin{lemma}\label{lm:star}
For $\mu\in X_+$ and for polarizable pure Hodge structures $\cH$ and $\cH'$ we have a canonical isomorphism
\begin{equation}\label{eq:star}
    \IC_\mu(\cH\otimes\cH')=\IC_0(\cH)\star\IC_\mu(\cH').
\end{equation}
\end{lemma}
\begin{proof}
  Consider the closed embedding
  \[
    i_0\colon\Gr_G=G_\cO\times^{G_\cO}\Gr_G\to G_\cK\times^{G_\cO}\Gr_G=\Gr_{conv}.
  \]
  It is easy to see that when $\lambda=0$ the morphism~\eqref{eq:jj} factors as $j_{0,\mu}=i_0\circ j_\mu$, where $j_\mu\colon\Gr^\mu\to\Gr_G$ is the embedding (note that $\Gr^{0,(m)}\times^{G_{(m)}}\Gr^\mu=\Gr^\mu$ and $a_{0,\mu}=a_\mu$). Thus, according to the previous lemma,
  \[
    \IC_0(\cH)\star\IC_\mu(\cH')=m_*(j_{0,\mu})_{!*}a_\mu^\dagger(\cH\otimes\cH')=
    m_*(i_0)_*(j_\mu)_{!*}a_\mu^\dagger(\cH\otimes\cH')=\IC_\mu(\cH\otimes\cH'),
  \]
  where we used that $m\circ i_0=\Id_{\Gr_G}$.
\end{proof}

It follows from Lemma~\ref{lm:star} and Proposition~\ref{pr:equiv} that the Hodge module $\delta_1:=\IC_0(\cH_0)$ is the identity for the convolution. Thus, we have equipped $\MH^{G_\cO}(\Gr_G)$ with a structure of monoidal category in the sense of~\cite[Def.~1.1]{DeligneMilneTannakian}.

\subsection{Rigidity} Let $\cF$ be an object of $\MH^{G_\cO}(\Gr_G)$. Recall from~\cite[Def.~2.1.1]{KirillovBakalov} that an object $\cF^\vee$ is the left dual of $\cF$ if there are morphisms $e_\cF\colon\cF^\vee\star\cF\to\delta_1$ (called the counit) and $i_\cF\colon\delta_1\to\cF\star\cF^\vee$ (called the unit) such that the composition
\[
    \cF\xrightarrow{i_\cF\otimes\Id_\cF}\cF\star\cF^\vee\star\cF\xrightarrow{\Id_\cF\otimes e_\cF}\cF
\]
is equal to $\Id_\cF$ and the composition
\[
    \cF^\vee\xrightarrow{\Id_{\cF^\vee}\otimes i_\cF}\cF^\vee\star\cF\star\cF^\vee\xrightarrow{\Id_{\cF^\vee}\otimes e_\cF}\cF^\vee
\]
is equal to $\Id_{\cF^\vee}$.
Right dual objects are defined similarly. Our goal is to show that $\MH^{G_\cO}(\Gr_G)$ is a rigid monoidal category, that is, every object has a left and a right dual.

Our first observation is that a monoidal functor obviously preserves duals. The functor $\MH(pt)\to\MH^{G_\cO}(\Gr_G)$ sending $\cH$ to $\IC_0(\cH)$ is easily seen to be monoidal, and the category of pure polarizable Hodge structures is rigid. Thus every object of the form $\IC_0(\cH)$ has duals.

Secondly, for $\lambda\in X_+$ set $\lambda':=w_0(-\lambda)\in X_+$, where $w_0$ is the longest element of the Weyl group of $G$. Set $\cF:=\IC_\lambda(\cH_0)$ and $\cF':=\IC_{\lambda'}(\cH_0)$. We claim that $\cF'(i)$, where $i=\dim\Gr^\lambda$, is the dual of $\cF$.

Indeed, first of all, we have $G_\cO\cdot(t^{-\lambda}\cdot G_\cO)=\Gr^{\lambda'}$. Now it follows from~\cite[(11.10)]{MirkovicVilonen} that we have $\Rat(\cF'(i))\simeq(\Rat(\cF))^\vee$ (note that for perverse sheaves we can trivialize the Tate twists). Next, we claim that the injective map
\begin{equation}\label{eq:RatIso}
    \Hom(\cF'(i)\star\cF,\delta_1)\hookrightarrow\Hom(\Rat(\cF'(i)\star\Rat(\cF),\Rat(\delta_1))
\end{equation}
is an isomorphism. Indeed, since the RHS is equal to $\Hom(\Rat(\cF),\Rat(\cF))$, it is a 1-dimensional vector space over $\Q$, so it is enough to show that the LHS is non-zero. However, we know that $\cF^\vee\star\cF$ is the direct sum of Hodge modules isomorphic to $\IC_\lambda(\cH)$, where $\cH$ are pure polarizable Hodge structures. Applying $\Rat$, and noting once again that the RHS of~\eqref{eq:RatIso} is a 1-dimensional vector space, we see that there is exactly one summand with $\lambda=0$ in this sum, and the rank of the corresponding Hodge structure $\cH$ is one. Thus $\cH=\cH_0(j)$ is Tate's Hodge structure.

It remains to compare weights: we have
\[
    \dim\Gr^{\lambda'}=(2\rho,\lambda')=(w_0^{-1}(2\rho),-\lambda)=(-2\rho,-\lambda)=\dim\Gr^\lambda.
\]
Thus, $\cF'(i)\star\cF$ has weight $\dim\Gr^{\lambda'}-2\dim\Gr^\lambda+\dim\Gr^\lambda=0$, so $j=0$. Thus the LHS of~\eqref{eq:RatIso} is non-zero, so~\eqref{eq:RatIso} is an isomorphism. We define the counit $e_\cF$ as the preimage of the counit for $\Rat(\cF)$ under this isomorphism.

We define the unit $i_\cF$ similarly. It follows from the fact that $\Rat$ is faithful and monoidal that $e_\cF$ and $i_\cF$ satisfy the above conditions for the unit and the counit. The existence of a right dual is proved similarly.

One checks that the dual of $\cF_1\star\cF_2$ is $\cF_2^\vee\star\cF_1^\vee$.
Thus, according to~\eqref{eq:star}, for every $\lambda\in X_+$ and for every pure polarizable Hodge structure $\cH$, the Hodge module $\IC_\lambda(\cH)$ has a dual. Finally, every object is the direct sum of the objects isomorphic to $\IC_\lambda(\cH)$, so every object has a dual and the category is proved to be rigid.

\subsection{Fiber functor}\label{sect:fiber}
We define the fiber functor $\Fib\colon\MH^{G_\cO}(\Gr_G)\to\Vect_\Q$ as the composition of $\Rat$ and the global cohomology functor from~\cite[Sect.~3]{MirkovicVilonen}. According to Proposition~\ref{pr:HM_basic_properties}(ii), Lemma~\ref{lm:RatMonoidal}(ii)), and~\cite[Cor.~3.7, Prop~6.2]{MirkovicVilonen}$\Fib$ is a faithful exact monoidal functor. We have equipped $\MH^{G_\cO}(\Gr_G)$ with a structure of a neutral Tannakian category.

\section{Geometric Satake equivalence for Hodge modules}\label{Sect:proof}
\subsection{Fiber functor to Hodge structures} First of all, we would like to upgrade the fiber functor to a functor to Hodge structures. Let $a$ denote the projection $\Gr_G\to pt$. Consider the functor
\[
   a_+\colon\MH^{G_\cO}(\Gr_G)\to\MH(pt)\colon\cF\mapsto\bigoplus_{i\in\Z}H^{2i}(a_*(\cF))(i).
\]
We note that $a_+$ takes $\MH^{G_\cO}(\Gr_G,n)$ to $\MH(pt,n)$. Here $\MH(pt)$ is the category of pure Hodge structures, $(i)$ stands for the $i$-th Tate twist of a pure Hodge structure. Note that by~\cite[Thm.~3.6]{MirkovicVilonen} the odd direct images vanish, so the composition of $a_+$ with $\Rat$ is the fiber functor defined in Section~\ref{sect:fiber}.

\begin{proposition}\label{pr:a_plus_monoidal}
    The functor $a_+$ is monoidal.
\end{proposition}
\begin{proof}
    Let $\cF_1,\cF_2\in\MH^{G_\cO}(\Gr_G)$, let $\pi$ and $j$ be as in diagram~\eqref{eq:CD}. Set
    \[
        \cG:=\pi_*j_{!*}(\cF_1\boxtimes\Q^{\mathrm{Hdg}}_{\A^1-0}\boxtimes\cF_2)\in D^b\MH(\A^1).
    \]
    \begin{lemma} Let $a\colon\Gr_G\to pt$ be the projection. We have
    \begin{equation}\label{eq:resG}
        \cG|_{\A^1-0}=a_*(\cF_1)\otimes a_*(\cF_2)\boxtimes\Q^{\mathrm{Hdg}}_{\A^1-0}.
    \end{equation}
    \end{lemma}
    \begin{proof}
      Using base change and Lemma~\ref{lm:j*composition}(i), we get
      \[
        \cG|_{\A^1-0}=p_{2,*}(\cF_1\boxtimes\Q^{\mathrm{Hdg}}_{\A^1-0}\boxtimes\cF_2),
      \]
      where $p_2\colon\Gr_G\times(\A^1-0)\times\Gr_G\to\A^1-0$ is the projection. Since $p_2=a\times\Id_{\A^1-0}\times a$, it remains to use the compatibility of $\boxtimes$ with direct images.
    \end{proof}
    Using the base change and Proposition~\ref{pr:Fusion}, we get
    \[
        \cG|_0=
        a_*(j_{!*}(\cF_1\boxtimes\Q^{\mathrm{Hdg}}_{\A^1-0}\boxtimes\cF_2)|_0)=a_*(\cF_1\star\cF_2).
    \]
    It follows from~\eqref{eq:resG},~\eqref{eq:support_decomposition}, and Proposition~\ref{pr:Saito3.21} that $\cG=a_*(\cF_1)\otimes a_*(\cF_2)\boxtimes\Q^{\mathrm{Hdg}}_{\A^1}\oplus\cG'$, where $\cG'\in D^b\MH(\A^1)$ is supported at zero. Thus we get a split injective homomorphism in $D^b\MH(pt)$:
    \[
        a_*(\cF_1)\otimes a_*(\cF_2)\hookrightarrow\cG|_0=a_*(\cF_1\star\cF_2).
    \]
    Applying the cohomology functor and using the K\"unneth decomposition (which follows from Proposition~\ref{pr:KashDb}(i)) we get an injective graded homomorphism
    \[
        a_+(\cF_1)\otimes a_+(\cF_2)\hookrightarrow a_+(\cF_1\star\cF_2).
    \]
    Since the global cohomology functor from~\cite{MirkovicVilonen} is monoidal, the graded dimensions are equal, so this homomorphism is an isomorphism.
\end{proof}

\subsection{The category of Tate's Hodge modules}
Recall that $\cH_0$ stands for the trivial 1-dimensional Hodge structure. The following proposition is crucial for the proofs of Theorems~\ref{th:HodgeSatake},~\ref{th:TateSatake},~and~\ref{th:MixedSatake}.

\begin{proposition}\label{pr:ClosedUnderStar}
Let $\Tate^{G_\cO}(\Gr_G)$ be the strictly full subcategory of $\MH^{G_\cO}(\Gr_G)$ whose objects are isomorphic to the direct sums of $\IC_\lambda(\cH_0(i))$, where $\lambda$ ranges over $X_+$ and $i$ ranges over $\Z$.

(i) For $\cF\in\MH^{G_\cO}(\Gr_G)$ the Hodge structure $a_+(\cF)$ is the direct sum of Tate twists of the trivial Hodge structure if and only if $\cF\in\Tate^{G_\cO}(\Gr_G)$.

(ii) The category $\Tate^{G_\cO}(\Gr_G)$ is closed under the convolution.
\end{proposition}
\begin{proof}
    (i) Using~\eqref{eq:star} and Lemma~\ref{pr:a_plus_monoidal}, we see that
    \[
        a_+(\IC_\lambda(\cH))=a_+(\IC_0(\cH)\star\IC_\lambda(\cH_0))=\cH\otimes a_+(\IC_\lambda(\cH_0)).
    \]
    Now the `only if' direction follows easily from Proposition~\ref{pr:equiv} and faithfulness of $a_+$.

    For the `if' direction, let $j_\lambda\colon\Gr^\lambda\to\Gr$ be the locally closed embedding and $a_\lambda\colon\Gr^\lambda\to pt$ be the constant morphism. Consider the mixed Hodge module $\cF_\lambda:=H^0((j_{\lambda})_!a_\lambda^\dagger\Q^{\mathrm{Hdg}})\in\MHM(\Gr_G)$. Recall that $a\colon\Gr_G\to pt$ is the constant morphism. We need a lemma.
\begin{lemma}\label{lm:StandardSheaf}
     For $i\in\Z$ we have $H^{2i+1}(a_*\cF_\lambda)=0$ and $H^{2i}(a_*\cF_\lambda)$ is a Tate twist of a trivial Hodge structure.
\end{lemma}
\begin{proof}
  Let $N\subset G$ be the unipotent radical of the Borel subgroup $B$. Let $\nu\in X_*$ be a co-character and let $S_\nu:=N_\cK\cdot(t^\nu G_\cO)$ be a semi-infinite Schubert cell (here $t^\nu$ is as in Section~\ref{sect:irr}). Let $\rho$ be the half sum of positive roots of $G$ with respect to $B$. Recall from~\cite[Thm.~3.2]{MirkovicVilonen} that the intersection $S_\nu\cap\Gr^\lambda$ is of pure dimension $\rho(\nu+\lambda)$ as long as it is non-empty.

  Let $a_{\nu,\lambda}\colon S_\nu\cap\Gr^\lambda\to pt$ be the projection to the point. Consider the top degree cohomology with compact support
  \[
        H_c^{2\rho(\nu+\lambda)}(S_\nu\cap\Gr^\lambda,\Q^{\mathrm{Hdg}}):=
        H^{2\rho(\nu+\lambda)}((a_{\nu,\lambda})_!a_{\nu,\lambda}^\dagger\Q^{\mathrm{Hdg}})
        \in\MHM(pt).
  \]
  Repeating the arguments of~\cite[Sect.~3]{MirkovicVilonen} (see especially Prop.~3.10 in loc.~cit.), we identify $H^i(a_*\cF_\lambda)$ with
  \[
    \bigoplus_{\nu\in X_*,2\rho(\nu)=i}H_c^{2\rho(\nu+\lambda)}(S_\nu\cap\Gr^\lambda,\Q).
  \]
  In particular, we see that $H^i(a_*\cF_\lambda)=0$ for odd $i$.

  It remains to show that $H_c^{2\rho(\nu+\lambda)}(S_\nu\cap\Gr^\lambda,\Q)$ is a Tate twist of a trivial Hodge structure. In fact, the top cohomology with compact support is always Tate. One way to show this is to reduce to the case when the scheme is irreducible and to note that the top cohomology is 1-dimensional in this case. 
\end{proof}
    We return to the proof of Proposition~\ref{pr:ClosedUnderStar}. Since $a_*=a_!$ preserves weights, we have $\gr^0(a_*\cF_\lambda)=a_*(\gr^0(\cF_\lambda))$. Now the previous lemma implies that $a_*(\gr^0(\cF_\lambda))$ is the direct sum of Tate twists of trivial Hodge structures.

    Note that $\gr^0\cF_\lambda\in\MH^0(\Gr_G)$ and it follows from semisimplicity that $\IC_\lambda(\cH_0)$ is
    its direct summand. Thus $a_*\IC_\lambda(\cH_0)$ is a direct summand of $a_*(\gr^0(\cF_\lambda))$ and the claim follows.

(ii) Since $a_+$ is monoidal (Proposition~\ref{pr:a_plus_monoidal}), part~(ii) follows from part~(i).
\end{proof}

\subsection{Proof of Theorem~\ref{th:HodgeSatake}}
We will construct an equivalence
\[
    \MH^{G_\cO}(\Gr_G)[\sqrt T]=\Perv_\Q(\Gr_G)\boxtimes\MH(pt)[\sqrt T].
\]
The theorem will follow by applying the usual geometric Satake equivalence $\Perv_\Q(\Gr_G)=\Rep\Q(\check G_\Q)$ and taking the even components of the both sides.

Consider the functor $F\colon\Perv_\Q(\Gr_G)\to\MH^{G_\cO}(\Gr_G)[\sqrt T]$ sending $\IC_\lambda$ to $\IC_\lambda(\cH_0)^+(\frac12\dim\Gr^\lambda)$. We claim that this is a monoidal functor. Indeed, take $\lambda,\mu\in X_*$.  By Proposition~\ref{pr:ClosedUnderStar}, $\IC_\lambda(\cH_0)(\frac12\dim\Gr^\lambda)\star\IC_\mu(\cH_0)(\frac12\dim\Gr^\mu)$ is the direct sum of objects of the form $\IC_\nu(\cH_0)(i)$, where $\nu\in X_*$ and $i\in\frac12\Z$. However, this object has weight zero, and the statement follows.

Next, we have an obvious embedding $G\colon\MH(pt)[\sqrt T]\to\MH^{G_\cO}(\Gr_G)[\sqrt T]$. We claim that $F\boxtimes G$ is an equivalence of categories. According to~\eqref{eq:star} and the commutativity constraints, we have
\[
    \IC_\lambda(\cH)^+=\IC_\lambda(\cH_0)^+\left(\frac12\dim\Gr^\lambda\right)\star
    \IC_0(\cH)^+\left(-\frac12\dim\Gr^\lambda\right).
\]
Similarly,
\[
    \IC_\lambda(\cH)^-=\IC_\lambda(\cH_0)^+\left(\frac12\dim\Gr^\lambda\right)\star
    \IC_0(\cH)^-\left(-\frac12\dim\Gr^\lambda\right).
\]
We see that by Proposition~\ref{pr:equiv}(i), the functor is essentially surjective. The fact that the functor is fully faithful follows from the definitions and Proposition~\ref{pr:equiv}(ii). \qed

\subsection{Proof of Theorem~\ref{th:TateSatake}}
Consider the functor $F\colon\Perv_\Q(\Gr_G)\to\Tate^{G_\cO}(\Gr_G)[\sqrt T]$ sending $\IC_\lambda$ to $\IC_\lambda(\cH_0)^+(\frac12\dim\Gr^\lambda)$. As in the proof of Theorem~\ref{th:HodgeSatake}, we check that this functor is monoidal.

Consider the category $\Rep\gm\Q$. Its irreducible objects are in a natural bijection with $\Z$, and we equip the category with the $\Z/2\Z$ parity grading. Consider the functor
\[
    G\colon\Rep\gm\Q\to\Tate^{G_\cO}(\Gr_G)[\sqrt T]
\]
sending the representation with weight $i\in\Z$ to $\IC(\cH_0)^+(i/2)$. As in the proof of Theorem~\ref{th:HodgeSatake}, we check that $F\boxtimes G$ is an equivalence of categories. Thus, using the usual geometric Satake, we get
\[
    \Tate^{G_\cO}(\Gr_G)[\sqrt T]=\Perv_\Q(\Gr_G)\boxtimes\Rep_\Q\gm\Q=\Rep_\Q(\check G_\Q\times\gm\Q).
\]
We note that the $\Z/2\Z$-grading on this category corresponds to a central embedding $\mu_2\to\check G_\Q\times\gm\Q$. It is easy to see that this embedding is exactly $2\rho\times\iota$. The theorem follows. \qed

\section{Mixed Hodge modules}\label{sect:Mixed} In this Section we sketch a proof of Theorem~\ref{th:MixedSatake}. In Section~\ref{sect:DerHodge} we defined the category $\MHM(Y)$ of mixed Hodge modules on an ind-scheme $Y$ with an exact faithful functor $\MHM(Y)\to\Perv_\Q(Y)$. We note that we have smooth descents for mixed Hodge modules, proof being the same as that of Proposition~\ref{pr:SmoothDescent}.

\subsection{Equivariant mixed Hodge modules} Let, as in Section~\ref{sect:EquivHodge}, $H$ be a connected pro-algebraic group acting nicely on an ind-scheme $Y$. In this case, using smooth descent for morphisms of mixed Hodge modules, we can define the category $\MHM^H(Y)$ of equivariant mixed Hodge modules. Similarly to Proposition~\ref{pr:ForgetEquiv}, we show that the forgetful functor $\MHM^H(Y)\to\MHM(Y)$ is fully faithful and we identify $\MHM^H(Y)$ with a full subcategory of $\MHM(Y)$. Finally, we have an analogue of Lemma~\ref{lm:EquivPushFrwrd}.

Recall that an object of $\MHM(Y)$ has a weight filtration. It follows from~\cite[Lemma~2.25]{SaitoMixed} that for an object in $\MHM^H(Y)$ this is a filtration by equivariant submodules so that the associated graded is in $\MH^H(Y)$.

\subsection{$G_\cO$-equivariant mixed Hodge modules on $\Gr_G$} By the above, we can consider the category $\MHM^{G_\cO}(\Gr_G)$. We want to describe the objects. Using the weight filtration and Proposition~\ref{pr:equiv}, we see that the simple objects are pure Hodge modules of the form $\IC_\lambda(\cH)$, where $\lambda\in X_+$ and $\cH$ is an irreducible Hodge structure.

The definition of $\IC_\lambda(\cH)$, given in Section~\ref{sect:irr}, generalizes to the mixed Hodge structures $\cH$ as follows: $\IC_\lambda(\cH):=\cH\otimes\IC_\lambda(\cH_0)$. (The original definition cannot be used because we do not define intermediate extensions for mixed Hodge modules.)

\begin{proposition}\label{pr:MHM_DirectSum}
Every object of $\MHM^{G_\cO}(\Gr_G)$ can be uniquely written as the direct sum of subobjects isomorphic to $\IC_\lambda(\cH)$ for some distinct co-characters $\lambda\in X_+$ and some mixed Hodge structures $\cH$.
\end{proposition}
\begin{proof}
Let $\lambda,\mu\in X_+$, $\lambda\ne\mu$, let $\cH$ and $\cH'$ be mixed Hodge structures. It is enough to show that
$\Ext^1(\IC_\lambda(\cH),\IC_\mu(\cH'))=0$. Without loss of generality we may assume that $\Gr^\mu\subset\overline\Gr^\lambda$. Let $j\colon\Gr^\mu\to\overline\Gr^\lambda$ be the locally closed embedding. It is enough to show that $j^*\IC_\lambda(\cH)$ lives in cohomological degrees $\le-2$. Since $\Rat$ is exact and faithful, it is enough to show that $j^*\IC_\lambda$ lives in perverse cohomological degrees $\le-2$. This is explained in \cite[Sect.~2.1.3]{GaitsgoryCentral}.
\end{proof}

We define the convolution of mixed Hodge modules similarly to the convolution of pure Hodge modules.

The constructions of rigidity, commutativity constraint, unit object, and the fiber functor are the same as in the pure case with minor modifications, thus $\MHM^{G_\cO}(\Gr_G)$ becomes a neutral Tannakian category.

Finally, as in the proof of Theorem~\ref{th:HodgeSatake} we construct an equivalence of categories.
\[
    \MHM^{G_\cO}(\Gr_G)[\sqrt T]=\Perv_\Q(\Gr_G)\boxtimes\MHM(pt)[\sqrt T].
\]
and the theorem follows. \qed

\bibliographystyle{../../alphanum}
\bibliography{../../RF}
\end{document}